\documentclass{amsart}

\usepackage{mathrsfs,amsmath,amssymb,amsfonts}

\usepackage[pdfdisplaydoctitle=true,colorlinks=true,urlcolor=blue,citecolor=blue,%
linkcolor=blue,pdfstartview=FitH,pdfpagemode=None,bookmarksnumbered=true]{hyperref}
\usepackage{cmap} 
\usepackage[usenames]{color}
\usepackage{leftidx}
\usepackage[normalem]{ulem}

\newtheorem{theorem}{Theorem}[section]
\newtheorem{lemma}[theorem]{Lemma}
\newtheorem{corollary}[theorem]{Corollary}
\newtheorem{proposition}[theorem]{Proposition}
\newtheorem{remark}[theorem]{Remark}
\newtheorem{definition}[theorem]{Definition}
\newtheorem{example}[theorem]{Example}

\numberwithin{equation}{section}

\newcommand{\bz}{{\mathbb B}}
\newcommand{\cz}{{\mathbb C}}
\newcommand{\gz}{{\mathbb Z}}
\newcommand{\nz}{{\mathbb N}}
\newcommand{\rz}{{\mathbb R}}

\newcommand{\calA}{\mathcal{A}}
\newcommand{\calD}{\mathcal{D}}

\newcommand{\calH}{\mathcal{H}}

\newcommand{\frakA}{\mathfrak{A}}

\newcommand{\scrA}{\mathscr{A}}
\newcommand{\scrB}{\mathscr{B}}
\newcommand{\scrC}{\mathscr{C}}

\newcommand{\scrF}{\mathscr{F}}
\newcommand{\scrH}{\mathscr{H}}
\newcommand{\scrL}{\mathscr{L}}
\newcommand{\scrP}{\mathscr{P}}
\newcommand{\scrS}{\mathscr{S}}

\newcommand{\ad}{\mathrm{ad}}
\newcommand{\aps}{{\mathrm{APS}}}
\newcommand{\cl}{\mathrm{cl}}

\newcommand{\dom}[1]{\operatorname{dom}\left( #1\right)}
\newcommand{\domain}{\mathrm{dom}}
\newcommand{\eps}{\varepsilon}
\newcommand{\forget}[1]{}

\newcommand{\lra}{\longrightarrow}

\newcommand{\pO}{{\partial\Omega}}
\newcommand{\pt}[1]{\left( #1 \right) }

\newcommand{\spk}[1]{\langle#1\rangle}
\newcommand{\wh}{\widehat}
\newcommand{\wt}{\widetilde}

\begin{document}

\title[APS-type boundary conditions and spectral triples]%
{Boundary value problems with Atiyah-Patodi- Singer type conditions and spectral triples}

\author{U.\ Battisti}
\address{Universit\`a di Torino, Dipartimento di Matematica, V. Carlo Alberto 10, 10123  Torino (Italy)}
\email{ubertino.battisti@unito.it,\  joerg.seiler@unito.it}

\author{J.\ Seiler}

\maketitle

\begin{abstract}
We study realizations of pseudodifferential operators acting on sections of vector-bundles on a smooth,  
compact manifold with boundary, subject to conditions of Atiyah-Patodi-Singer type. 
Ellipticity and Fredholm property, compositions, adjoints and self-adjoint\-ness of such realizations are 
discussed. We construct regular spectral triples $(\calA,\calH,\calD)$ for manifolds with boundary of 
arbitrary dimension, where $\calH$ is the space of square integrable sections. 
Starting out from Dirac operators with APS-conditions, these triples are even in 
case of even dimensional manifolds; we show that the closure of $\calA$ in $\scrL(\calH)$ coincides with 
the continuous functions on the manifold being constant on each connected component of the boundary.   
\end{abstract}



\section{Introduction}\label{sec:intro}

Spectral triples play a fundamental role in non commutative geometry
and provide a new approach to several fields in mathematics and physics.
One of the most striking results involving spectral triples is Connes' famous reconstruction Theorem 
\cite{Co13}, which shows that one can $($re-$)$construct from a commutative spectral triple 
$(\calA,\calH,\calD)$, satisfying certain 
axioms,  a compact oriented manifold \textit{without boundary} $M$ such that $\calA$ is isomorphic to 
$\scrC^{\infty}(M)$.
In the past years, the definition of spectral triple has been extended to different settings.
For example by Lescure \cite{Le01} to manifolds with conical singularities, by 
Lapidus \cite{LA97}, Cipriani et al. \cite{CG14}, 
and Christensen et al. \cite{CIS12} to fractals. Our paper provides a contribution to the analysis of 
spectral triples for manifolds 
with $($smooth$)$ boundary, mainly motivated by the recent work \cite{IoLe} of Iochum and Levy.

The central analytic tool our approach relies on is Boutet de Monvel's algebra of pseudodifferential 
boundary value problems 
\cite{Bout}, respectively a suitable extension of it going back to Schulze \cite{Schu37}, cf.\ also Seiler \cite{Seil}. 
This calculus provides an efficient framework for the application of microlocal methods in partial differential 
equations, geometric 
analysis and index theory for manifolds with boundary. We shall use this calculus for a systematic study 
of \textit{realizations} 
$($i.e., closed extensions$)$ of $($pseudo$)$differential operators on compact manifolds subject to 
homogeneous boundary 
conditions. This study is inspired by and extends work of Grubb \cite{Grub}. In comparison to her results, 
we allow a wider class of 
boundary conditions which we named \textit{APS-type conditions}, since the classical spectral boundary conditions 
of Atiyah-Patodi-Singer \cite{APS} are a particular example of such conditions. Specifically, these boundary conditions 
are of the form 
 $$\scrC^\infty(\Omega,E)\lra \scrC^\infty(\pO,F),\qquad u\mapsto Tu:=P(S\rho +T^\prime)u,$$
where $\Omega$ is a smooth Riemannian manifold with boundary, $E$ is a hermitian vector bundle over $\Omega$, 
$F=F_0\oplus\ldots\oplus F_{d-1}$ with $F_j$ hermitian vector bundles over $\pO$ $($possibly zero-dimensional$)$, 
$T^\prime=(T_0^\prime,\ldots,T_{d-1}^\prime)$ with trace operators 
$T^\prime_k:\scrC^\infty(\Omega,E)\to\scrC^\infty(\pO,F_j)$ of order $j+1/2$, 
$\rho=(\gamma_0,\ldots,\gamma_{d-1})^t$ with $\gamma_j$ denoting the operator of restriction to the boundary 
of the $j$-th derivative in direction normal to the boundary, $S=(S_{jk})_{0\le j,k\le d-1}$ with 
$S_{jk}\in L^{j-k}_\cl(\partial\Omega;E|_{\pO},F_j)$ being classical $($i.e., step one poly-homogeneous$)$ 
pseudodifferential operators of order $j-k$ on the boundary, and an \textit{idempotent}  
$P=(P_{jk})_{0\le j,k\le d-1}$ with $P_{jk}\in L^{j-k}_\cl(\partial\Omega;F_k,F_j)$. We then 
consider operators with domain $\{u\in H^d(\Omega,E)\mid Tu=0\}$ and with action given by a $d$-th order 
operator from Boutet de Monvel's calculus acting between sections of $E$. In Section \ref{sec:real} we discuss  
ellipticity and Fredholm property, the adjoint $($in particular, self-adjointness$)$ and composition of such realizations. 

In this context we prove and make use of a result on the invariance of the Fredholm index and the existence 
of inverses $($parametrices$)$ modulo projections onto the kernel for operators acting in families of Banach 
spaces, generalizing known, analogous results for pseudodifferential operator algebras to an abstract setting. 
This result is of independent interest as it applies to any operator algebra satisfying some very natural conditions, 
and is presented in the Appendix.  

The framework developed in Section \ref{sec:real} allows us to introduce and analyze, in Section \ref{sec:spectral}, 
spectral triples for manifolds with boundary. At a first glance, the approach is very similar to that of 
Iochum and Levy \cite{IoLe}, 
however it provides a true extension of their results. 
The main example of \cite{IoLe} are spectral triples based on Dirac 
operators equipped with chiral boundary conditions; there are good physical and mathematical motivations to consider 
this kind of boundary conditions, as it has been already done in several other works, see \cite{BG92} and \cite{CC07} 
for example. Being local conditions, Iochum and Levy could rely on the results of Grubb \cite{Grub} mentioned above. 
However, it is well known that chiral boundary condition cannot be defined in all settings. 
Indeed, it is always possible only in case the underlying manifold is of even dimension, in general a chirality operator 
is not naturally defined. In view of this lack of generality it seems natural to 
make use of \emph{non-local} APS-type boundary conditions and, in fact, this is what our approach permits to do. 
We show how to define \textit{regular} spectral triples $\pt{\calA^{\infty}_\calD, \calH, \calD}$ on every 
compact manifold with boundary, including, as particular example, those triples starting from Dirac operators 
equipped with APS boundary conditions. In the case of even-dimensional manifolds we show that the latter 
spectral triples respect the natural grading defined by the chirality, and therefore define so-called 
\textit{even} spectral triples, see Remark \ref{rem:evenspectral}.

Analogously to the case of chiral boundary conditions, the algebra $\calA_{\calD}^{\infty}$ is not the whole space 
$\scrC^{\infty}(\Omega)$, but a true subalgebra. In general, it is difficult to describe this algebra in explicit terms, 
but, in the case of a Dirac operator with APS conditions, we prove that the \textit{closure} of 
$\calA_{\calD}^{\infty}$ with respect to the supremum norm is the $C^*$-algebra of continuous functions 
being constant on each connected component of the boundary. The knowledge of this closure is important, 
since it plays a key role in Connes' reconstruction Theorem. In this context, our result implies that the spectral 
triple does \textit{not} fulfill the so-called Finiteness Axiom in \cite{Co13}, cf.\ Section \ref{sec:exe}. Roughly 
speaking, $\calA^{\infty}_\calD$ results to be \emph{too small} to see the geometric properties of the boundary. 
This kind of negative result actually indicates that the correct notion of spectral triple able to reconstruct 
manifolds with boundary, taking properly into account the geometry of the boundary, still has to be found. 
For the time being we have to leave this as an open problem for future research.

\textbf{Convention:} Throughout the text, we denote by $\Omega$ a smooth, compact, Riemannian manifold with 
boundary. On a collar-neighborhood $U$ of the boundary, identified with $\partial\Omega\times[0,\eps)$ 
and using the splitting of variables $x=(x^\prime,x_n)$, we assume the metric to be of product-form 
$g_{\partial\Omega}+dx_n^2$. Vector bundles over $\Omega$ mean smooth, hermitian vector-bundles that respect 
the product structure near the boundary, i.e., if $E$ denotes such a bundle, then $E|_U=\pi^* E|_{\partial\Omega}$ 
with $\pi(x^\prime,x_n)=x^\prime$ the canonical projection of the collar-neighborhood onto the boundary. In writing 
$\scrC^\infty(\Omega)$ we mean functions smooth up to $($i.e., including$)$ the boundary. 

\section{Boutet de Monvel's calculus for Toeplitz type operators}\label{sec:BdM}

Boutet de Monvel's algebra for boundary value problems on $\Omega$ consists of certain operators in block-matrix form,  
 \begin{equation}\label{eq:bdm01}
  \scrA=
  \begin{pmatrix}
   A_++G & K \\
   T & Q
  \end{pmatrix}\;:\quad
  \begin{matrix}
   \scrC^\infty(\Omega,E_0)\\
   \oplus\\
   \scrC^\infty(\pO,F_0)
  \end{matrix}
  \longrightarrow
  \begin{matrix}
   \scrC^\infty(\Omega,E_1)\\
   \oplus\\
   \scrC^\infty(\pO,F_1)
  \end{matrix}, 
 \end{equation}
where $E_j$ and $F_j$ are vector bundles over $\Omega$ and $\pO$, respectively, which are  allowed to be zero 
dimensional. Every such operator has an order, denoted by $\mu\in\gz$, and a type, denoted by 
$d\in\nz_0$.\footnote{It is possible to introcduce operators with negative type, cf.\ \cite{Grub}. However, for our 
purpose it is sufficient to consider non-negative types only.} 
To fix some terminology, 
\begin{itemize}
 \item $A_+$ is the ``restriction" to the interior of $\Omega$ of a $\mu$-th order, classical pseudodifferential 
  operator $A$ defined on the smooth double $2\Omega$,  having the $($two-sided$)$ transmission property with respect to 
  $\pO$,\footnote{$A_+=r_+ Ae_+$, where $r_+$ denotes the operator of restricting distributions from $2\Omega$ to 
  $\mathrm{int}\,\Omega$ and $e_+$ denotes the operator of extending $($sufficiently regular$)$ distributions by 0 from 
  $\mathrm{int}\,\Omega$ to $2\Omega$. If $A$ is differential, $A_+$ coincides with the standard action of $A$ on distributions; 
  occasionally we will therefore drop the subscript $+$ when dealing with differential operators.}  
 \item $G$ is a singular Green operator of order $\mu$ and type $d$, 
 \item $K$ is a $\mu$-th order potential operator, 
 \item $T$ is a trace operator of order $\mu$ and type $d$, 
 \item $Q$ is a $\mu$-th order, classical pseudodifferential operator on the boundary $\pO$. 
\end{itemize}
The space of all such operators we shall denote by $\scrB^{\mu,d}(\Omega;(E_0,F_0),(E_1,F_1))$. 

As a matter of fact, with $\scrA$ is associated a 
$($homogeneous$)$ principal symbol 
 $$\sigma^\mu(\scrA)=\big(\sigma^\mu_\psi(\scrA),\sigma^\mu_\partial(\scrA)\big),$$
where 
 $$\sigma^\mu_\psi(\scrA)=\sigma^\mu_\psi(A):\pi_\Omega^*E_0\lra\pi_\Omega^*E_1$$
is the usual principal symbol of the pseudodifferential operator $A$ $($restricted to $T^*\Omega)$, while 
 $$\sigma^\mu_\partial(\scrA):
     \begin{matrix}\pi_\pO^*(\scrS(\rz_+)\otimes E_0)\\ \oplus\\ \pi_\pO^* F_0\end{matrix}
     \lra
     \begin{matrix}\pi_\pO^*(\scrS(\rz_+)\otimes E_1)\\ \oplus\\ \pi_\pO^* F_1\end{matrix}$$
is the so-called principal boundary symbol; here $\pi_M:T^*M\to M$ denotes the canonical projection of the 
tangent bundle to the manifold and $\pi_M^*$ indicates pull-back of vector-bundles and 
$\scrS(\rz_+)\otimes E$ denotes the bundle with fibre $\scrS(\rz_+,E_y)$ in $y\in\pO$. 

For convenience of the reader, in the following subsection we shall shortly describe the above mentioned structures 
in the model case of $\Omega$ being a half-space and the bundles involved being trivial one-dimensional. 
For more complete descriptions we refer the reader to the existing literature on Boutet de Monvel's calculus, 
for instance \cite{Bout}, \cite{Grub}, \cite{ReSc}, and \cite{Schr}.

\subsection{A few details on the structure of the operators}\label{sec:BdM01}
Let $\Omega=\rz^{n-1}\times(0,+\infty)$ with variable $x=(x^\prime,x_n)$ and corresponding co-variable 
$\xi=(\xi^\prime,\xi_n)$. 
With $[\,\cdot\,]$ denote a smooth, positive function that coincides with the Euclidean norm outside
a neighborhood of the origin. Let 
 $$k(x^\prime,\xi^\prime;y_n)=k_0(x^\prime,\xi^\prime;[\xi^\prime] y_n),$$
where $k_0(x^\prime,\xi^\prime;t)$ behaves like a classical pseudodifferential symbol  of order 
$\mu+1/2$ in the variables 
$(x^\prime,\xi^\prime)$, while in $t$ like a rapidly decreasing function $($smooth up to $t=0)$. Then 
 $$K\varphi(x^\prime,x_n)
     = (2\pi)^{-n+1}\int_{\rz^{n-1}} e^{i x^\prime \xi^\prime}k(x^\prime, \xi^\prime; x_{n}) \,
         \scrF\varphi(\xi^\prime)  \, d\xi^\prime $$
defines a Poisson operator of order $\mu$, while  
 $$T_0 u(x^\prime) = (2\pi)^{-n+1} \int_{\rz^{n-1}}\int_{0}^\infty 
     e^{i x^\prime \xi^\prime}k(x^\prime, \xi^\prime; y_{n}) \,
    \scrF_{y^\prime\to\xi^\prime}u(\xi^\prime, y_{n})  \, dy_{n} d\xi^\prime $$ 
defines a trace operator of order $\mu$ and type $0$ $($note that taking formal adjoints with respect to the 
corresponding $L_2$-scalar products gives a one-to-one correspondence between these two types of operators$)$. 
A trace operator of order $\mu$ and type $d$ is of the form 
 $$Tu = \sum_{j=0}^{d-1} S_{j}\left(\left.\frac{d^ju}{dx_n^j}\right|_{x_n=0}\right) + T_0u $$ 
with classical pseudodifferential operators $S_{j}$ of order $\mu-j-1/2$ on the boundary $\rz^{n-1}$. 
A singular Green operator of order $\mu$ and type $0$ has the form 
 $$G_0 u(x^\prime,x_n) = (2\pi)^{-n+1} \int_{\rz^{n-1}}\int_{0}^\infty 
     e^{i x^\prime \xi^\prime}k(x^\prime, \xi^\prime; x_{n},y_n) \,
    \scrF_{y^\prime\to\xi^\prime}u(\xi^\prime, y_{n})  \, dy_{n} d\xi^\prime, $$ 
where  
 $$k(x^\prime,\xi^\prime;x_n,y_n)=g_0(x^\prime,\xi^\prime;[\xi^\prime] x_n,[\xi^\prime] y_n)$$
with a function $g_0(x^\prime,\xi^\prime;s,t)$ that behaves like a classical pseudodifferential symbol  of order 
$\mu+1$ in  $(x^\prime,\xi^\prime)$, while in $(s,t)$ like a rapidly decreasing function $($and smooth up to $s=0$ 
and $t=0)$. A singular Green operator of order $\mu$ and type $d$ is then of the form 
 $$Gu = \sum_{j=0}^{d-1} K_{j}\left(\left.\frac{d^ju}{dx_n^j}\right|_{x_n=0}\right) + G_0u $$ 
with Poisson operators $K_{j}$ of order $\mu-j-1/2$. 

The corresponding principal boundary symbols are defined as 
\begin{align*}
 \sigma^\mu_\partial&(K)(x^\prime,\xi^\prime):\cz\lra\scrS(\rz_+),\quad 
     c\mapsto c\,k_0^{(\mu+1/2)}(x^\prime,\xi^\prime;|\xi^\prime|\cdot)\\
 \sigma^\mu_\partial&(T_0)(x^\prime,\xi^\prime):\scrS(\rz_+)\lra\cz,\quad 
     u\mapsto \int_0^\infty k_0^{(\mu+1/2)}(x^\prime,\xi^\prime;|\xi^\prime|y_n)u(y_n)\,dy_n
\end{align*}
for potential and trace operators of type $0$, where $k_0^{(\mu+1/2)}$ denotes the homogeneous 
principal symbol of $k_0$ with respect to $(x^\prime,\xi^\prime)$. Moreover, 
\begin{align*}
 \sigma^\mu_\partial(T)(x^\prime,\xi^\prime)u=
 \sum_{j=0}^{d-1} \sigma_\psi^{\mu-j-1/2}(S_j)(x^\prime,\xi^\prime)\frac{d^ju}{dx_n^j}(0) + 
 \sigma^\mu_\partial(T_0)(x^\prime,\xi^\prime)u.
\end{align*}
Concerning the singular Green operators, we similarly have 
\begin{align*}
 \sigma^\mu_\partial(G_0)(x^\prime,\xi^\prime):\scrS(\rz_+)\lra\scrS(\rz_+),\quad 
 u\mapsto \int_0^\infty g_0^{(\mu+1)}(x^\prime,\xi^\prime;|\xi^\prime|\cdot,|\xi^\prime|y_n)u(y_n)\,dy_n,
\end{align*}
and
 $$\sigma^\mu_\partial(G)(x^\prime,\xi^\prime)u 
     = \sum_{j=0}^{d-1} \sigma^{\mu-j-1/2}_\partial(K_{j})(x^\prime,\xi^\prime)\frac{d^ju}{dx_n^j}(0) 
     + \sigma^\mu_\partial(G_0)(x^\prime,\xi^\prime)u. $$ 
\subsection{Basic properties of Boutet's calculus}\label{sec:BdM02}

The above described class of operators forms an ``algebra" in the sense that composition of operators induces 
maps  
\begin{align*}
 \scrB^{\mu_1,d_1}&(\Omega;(E_1,F_1),(E_2,F_2))\times \scrB^{\mu_0,d_0}(\Omega;(E_0,F_0),(E_1,F_1))\\
 &\lra \scrB^{\mu,d}(\Omega;(E_0,F_0),(E_2,F_2)),
\end{align*}
where the resulting order and type are  
 $$\mu=\mu_0+\mu_1,\qquad  d=\max(d_0,d_1+\mu_0).$$
The operators, initially acting on smooth sections, extend by density and continuity to Sobolev spaces, i.e., 
$\scrA\in\scrB^{\mu,d}(\Omega;(E_0,F_0),(E_1,F_1))$ induces maps 
\begin{equation}\label{eq:cont_ext1}
 \begin{matrix}
       H^s_p(\Omega,E_0)\\ \oplus\\ B^{s-(\frac{1}{p}-\frac{1}{2})}_{pp}(\partial\Omega,F_0)
     \end{matrix}
     \lra
     \begin{matrix}
       H^{s-\mu}_p(\Omega,E_0)\\ \oplus\\ B^{s-\mu-(\frac{1}{p}-\frac{1}{2})}_{pp}(\partial\Omega,F_1)
     \end{matrix},
     \qquad s>d-1+\frac{1}{p},
\end{equation}
where $1<p<\infty$ and $H^s_p$ denotes the standard Sobolev $($Bessel potential$)$ spaces, while 
$B^s_{pq}$ are the usual Besov spaces. Similarly, the boundary symbol extends to maps 
\begin{equation}\label{eq:cont_ext2}
 \begin{matrix}\pi_\pO^*(H^s_p(\rz_+)\otimes E_0)\\ \oplus\\ \pi_\pO^* F_0\end{matrix}
  \lra
 \begin{matrix}\pi_\pO^*(H^{s-\mu}_p(\rz_+)\otimes E_1)\\ \oplus\\ \pi_\pO^* F_1\end{matrix}.
\end{equation}
We shall employ these properties only in the Hilbert space case $p=2$; 
in this case $B^s_{pp}=H^s_2$ and we eliminate the index $p=2$ from the notation. 

\subsection{Toeplitz type operators and ellipticity}

In this paper we shall need an extended version of Boutet de Monvel's calculus. As described here, this calculus was 
introduced in \cite{Schu37}; it can be also obtained as a special case from a general approach to operator-algebras of 
Toeplitz type developed in \cite{Seil}.  

Let $P_j\in L^0_\cl(\pO;F_j,F_j)$, $j=0,1$, be two pseudodifferential projections on the boundary of $\Omega$. 
We then denote by 
 $$\scrB^{\mu,d}(\Omega;(E_0,F_0;P_0),(E_1,F_1;P_1))$$
the space of all operators $\scrA\in\scrB^{\mu,d}(\Omega;(E_0,F_0),(E_1,F_1))$ such that 
 $$\scrA(1-\scrP_0)=(1-\scrP_1)\scrA=0,\qquad \scrP_j:=\begin{pmatrix}1&0\\0&P_j\end{pmatrix}.$$
Being projections, the range spaces $H^s(\pO,F_j,P_j):=P_j\big(H^s(\pO,F_j)\big)$ are closed subspaces of 
$H^s(\pO,F_j)$, and any such $\scrA$ induces continuous maps 
\begin{equation}\label{eq:cont_ext3}
 \begin{matrix}
       H^s(\Omega,E_0)\\ \oplus\\ H^{s}(\partial\Omega,F_0,P_0)
     \end{matrix}
     \lra
     \begin{matrix}
       H^{s-\mu}(\Omega,E_0)\\ \oplus\\ H^{s-\mu}(\partial\Omega,F_1,P_1)
     \end{matrix},
     \qquad s>d-\frac{1}{2}, 
\end{equation}
according to \eqref{eq:cont_ext1}. With $P_j$ also the principal symbols $\sigma^0_\psi(P_j)$ are projections 
$($as bundle morphisms$)$ and thus define a subbundle $F_j(P_j)$ of $\pi_\pO^*F_j$. We then set 
 $$\sigma^\mu(\scrA;P_0,P_1):=\left(\sigma_\psi^\mu(\scrA),\sigma_\partial^\mu(\scrA;P_0,P_1)\right)$$
with $\sigma_\partial^\mu(\scrA;P_0,P_1)$ being the principal boundary symbol of $\scrA$ considered as a map 
\begin{equation}\label{eq:cont_ext4}
 \begin{matrix}\pi_\pO^*(H^s(\rz_+)\otimes E_0)\\ \oplus\\ F_0(P_0)\end{matrix}
  \lra
 \begin{matrix}\pi_\pO^*(H^{s-\mu}(\rz_+)\otimes E_1)\\ \oplus\\ F_1(P_1)\end{matrix},
 \qquad s>d-\frac{1}{2}, 
\end{equation}
$($or, alternatively, replacing the Sobolev spaces by $\scrS(\rz_+))$. 

\begin{definition}
$\scrA\in \scrB^{\mu,d}(\Omega;(E_0,F_0;P_0),(E_1,F_1;P_1))$ is called elliptic if both components of the 
principal symbol $\sigma^\mu(\scrA;P_0,P_1)$ are isomorphisms.\footnote{Invertibility of the principal boundary 
symbol as a map \eqref{eq:cont_ext4} is independent of the choice of $s$ and, equivalently, one may replace the 
Sobolev spaces by $\scrS(\rz_+)$.} 
\end{definition}

The following result is the main theorem of elliptic theory of Toeplitz type operators. For details see Section 2.1 of 
\cite{Schu37} and Theorem 6.1 of \cite{Seil}. 

\begin{theorem}\label{thm:main_toeplitz}
For $\scrA_0\in\scrB^{\mu,d}(\Omega;(E_0,F_0;P_0),(E_1,F_1;P_1))$ the following statements are equivalent$:$
\begin{enumerate}
 \item $\scrA_0$ is elliptic. 
 \item There exists an $s>\max(\mu,d)-1/2$ such that the map \eqref{eq:cont_ext3} associated with $\scrA_0$ is 
  Fredholm. 
 \item For every $s>\max(\mu,d)-1/2$ the map \eqref{eq:cont_ext3} associated with $\scrA_0$ is Fredholm. 
 \item There is an $\scrA_1\in \scrB^{-\mu_0,\max(d-\mu,0)}(\Omega;(E_1,F_1;P_1),(E_0,F_0;P_0))$  
  such that 
  \begin{align*}
   \scrA_1\scrA_0-\scrP_0 &\in \scrB^{-\infty,\max(\mu,d)}(\Omega;(E_0,F_0;P_0),(E_0,F_0;P_0)),\\
   \scrA_0\scrA_1-\scrP_1 &\in \scrB^{-\infty,\max(d-\mu,0)}(\Omega;(E_1,F_1;P_1),(E_1,F_1;P_1)).
  \end{align*}
\end{enumerate}
Any such operator $\scrA_1$ is called a parametrix of $\scrA_0$.  
\end{theorem}

\section{Realizations subject to APS-type boundary conditions}\label{sec:real}

In this section we shall study certain closed extensions of unbounded operators of the form 
 $$A_++G:\scrC^\infty(\Omega,E)\subset L_2(\Omega,E)\lra L_2(\Omega,E)$$
with $A_++G\in B^{d,d}(\Omega;E,E):=\scrB^{d,d}(\Omega;(E,0;1),(E,0;1))$, 
subject to $($a vector of$)$ boundary conditions of APS-type, 
which we shall describe in the following subsection. Our results extend those of Sections 1.4 and 1.6 of 
\cite{Grub}; for convenience of the reader we shall employ similar notation.  

\subsection{APS-type boundary conditions}\label{sec:real01}

Let $d\in\nz$ be a positive integer and let $\partial/\partial\nu$ denote the derivative in direction of the outer 
normal to $\partial\Omega$. We define, for $s>d+j-\frac{1}{2}$, 
 $$\gamma_j:H^s(\Omega,E)\to H^{s-j-\frac{1}{2}}(\pO,E|_{\pO}),\qquad 
     u\mapsto \frac{\partial^ju}{\partial\nu^j}\Big|_{\pO},$$
and $\rho=(\gamma_0,\ldots,\gamma_{d-1})^t$, where $E|_{\pO}$ indicates the restriction of the bundle 
$E$ to the boundary. Moreover,  
\begin{equation}\label{eq:bc}
  T_j=\sum_{k=0}^{d-1} S_{jk}\gamma_k+T^\prime_j:H^{s}(\Omega,E)\lra 
  H^{s-j-\frac{1}{2}}(\pO,F_j), 
\end{equation}
with vector bundles $F_j$ over $\partial\Omega$ $($possibly zero-dimensional$)$, pseudodifferential operators 
$S_{jk}\in L^{j-k}_\cl(\partial\Omega;E|_{\pO},F_j)$ and trace operators $T_j^\prime$ of order $j+1/2$ 
and type 0. We write $T^\prime=(T_0^\prime,\ldots,T_{d-1}^\prime)^t$ and further introduce 
\begin{align*}
  \scrH^s(\pO,E)&=\mathop{\oplus}_{j=0}^{d-1} H^{s+d-j-\frac{1}{2}}(\pO,E|_{\partial\Omega}),\\ 
  \scrH^s(\pO,F)&=\mathop{\oplus}_{j=0}^{d-1} H^{s+d-j-\frac{1}{2}}(\pO,F_k).
\end{align*}

\begin{definition}\label{def:aps}
Using the previously introduced notation, an APS-type boundary condition $T$ is an operator of the form
 $$T=P(S\rho+T^\prime):H^{s}(\Omega,E)\lra \scrH^{s-d}(\partial\Omega,F),\qquad s>d-1/2,$$ 
where $S=(S_{jk})_{0\le j,k\le d-1}$ and a projection $($i.e., idempotent$)$  
 $$P=(P_{jk})_{0\le j,k\le d-1}\quad\text{with}\quad P_{jk}\in L^{j-k}_\cl(\partial\Omega;F_k,F_j).$$
\end{definition}

To give an example, 
let $B_j$ be a pseudodifferential operator of integer order $0\le j<d$ on the double of $\Omega$ satisfying 
the transmission condition with respect to $\pO$ and $T_j:=\gamma_0\circ B_{j,+}$. Then $T_j$ is as in 
\eqref{eq:bc}, even with $S_{jk}=0$ for $k>j$ and all $S_{jj}$ are zero-order differential operators, i.e., 
induced by a bundle homomorphism $s_{j}:E|_{\pO}\to F_j$. Hence, $T=S\rho+T^\prime$ with a left-lower 
triangular matrix $S$ whose diagonal elements are zero-order differential. 

The classical Atiyah-Patodi-Singer conditions are included in this setting by taking $d=1$, $T^\prime=0$ and $S$
equal to the identity in Definition \ref{def:aps}. 

\begin{definition}\label{def:aps2}
Let $A_++G\in B^{d,d}(\Omega;E,E)$ and $T$ be an APS-type boundary con\-dition as described above. 
We write $(A_++G)_T$ for the operator acting like $A_++G$ on the domain 
 $$\mathrm{dom}((A_++G)_T)=\Big\{u\in H^d(\Omega,E) \mid Tu=0\Big\}.$$ 
\end{definition}

The operator $(A_++G)_T$ is often called the realization of $A_++G$ subject to the boundary condition $T$. 
We call two boundary conditions $T_0$ and $T_1$ equivalent, if they have the same kernel as maps on 
$H^d(\Omega,E)$; then, obviously, $(A_++G)_{T_0}=(A_++G)_{T_1}$. 

\subsection{Elliptic and normal realizations}\label{sec:real02}

Now let $\Lambda=\mathrm{diag}(\Lambda_0,\ldots,\Lambda_{d-1})$ be a $(d\times d)$-diagonal matrix 
with invertible components $\Lambda_j\in L^{d-j-\frac{1}{2}}_\cl(\partial\Omega;F_j,F_j)$. Note that then 
 $$P_\Lambda:=\Lambda P\Lambda^{-1}\in L^0_\cl(\pO;F_\partial,F_\partial), \qquad 
     F_\partial:=F_0\oplus\ldots\oplus F_{d-1},$$
is a zero order projection. 

\begin{definition}\label{def:ell+norm}
Consider the realization $(A_++G)_T$ with $T=P(S\rho+T^\prime)$. 
\begin{enumerate}
 \item The realization is called elliptic if 
   $$\begin{pmatrix}A_++G\\ \Lambda T\end{pmatrix}
       =\begin{pmatrix}1&0\\0&P_\Lambda\end{pmatrix}
       \begin{pmatrix}A_++G\\ \Lambda(S\rho+T^\prime)\end{pmatrix}$$ 
  is an elliptic element in $\scrB^{d,d}(\Omega;(E,0;1),(E,F_\partial;P_\Lambda))$. 
 \item The boundary condition $T$ is called normal if there exists a matrix 
   $$R=(R_{jk})_{0\le j,k\le d-1},\qquad R_{jk}\in L^{j-k}_\cl(\pO;F_k,E|_{\pO}),$$
  such that $PSR=P$. As way of speaking, we occasionally will call $R$ the right-inverse of $PS$. 
\end{enumerate}
\end{definition}

Note that ellipticity of $(A_++G)_T$ is equivalent to the Fredholm property of 
 $$\begin{pmatrix}A_++G\\ T\end{pmatrix}:
     H^s(\Omega,E)\lra
     \begin{matrix}H^{s-d}(\Omega,E)\\ \oplus\\ \scrH^{s-d}(\pO,F,P)\end{matrix}$$
for some $($and then for all$)$ $s>d-1/2$, where, by definition, 
 $$\scrH^{s}(\pO,F,P)=P\big(\scrH^{s}(\pO,F)\big).$$ 
By abstract and well-known results on Fredholm operators (see, for example, Theorem 8.3 in \cite{CSS}), 
this in turn is equivalent to the Fredholm property of 
 $$(A_++G)_T:\mathrm{dom}((A_++G)_T)\to L^2(\Omega,E)$$ 
together with the finiteness of the codimension of $T(H^d(\Omega,E))$ in $\scrH^{0}(\pO,F,P)$. 

It is useful to observe that realizations with a normal boundary condition can be represented in a certain 
canonical form: If $T=P(S\rho+T^\prime)$ is normal as in Definition \ref{def:ell+norm}, then 
$\wt{T}:=RT$ is a boundary condition equivalent to $T$ in view of the injectivity of $R$ on the range of $P$. 
Moreover, 
\begin{equation}\label{eq:canonical}
  \wt{T}=\wt{P}(\rho+\wt{T}^\prime),\qquad \wt{P}=RPS,\qquad \wt{T}^\prime=RT^\prime,
\end{equation}
where $\wt{P}$ is a projection with components $\wt{P}_{jk}\in L^{j-k}_\cl(\pO;E|_{\pO},E|_{\pO})$ and 
trace operators $\wt{T}^\prime_j:H^{s}(\Omega,E)\to H^{s-j-\frac{1}{2}}(\pO,E|_{\pO})$  of order 
$j+1/2$ and type 0. 

\begin{lemma}\label{lem:surj}
A normal boundary condition $T=P(S\rho+T^\prime)$ induces surjective maps 
$H^s(\Omega,E)\to\scrH^{s-d}(\pO,F,P)$, $s>d-1/2$. 
\end{lemma}
\begin{proof}
With the previously introduced notation, $T=PS(\rho+R{T}^\prime)$. 
By Proposition 1.6.5 of \cite{Grub} we know that 
$\rho+R{T}^\prime:H^s(\Omega,E)\to\scrH^{s-d}(\pO,E)$ is surjective. It remains to observe that 
$PS:\scrH^{s-d}(\pO,E)\to\scrH^{s-d}(\pO,F,P)$ surjectively, due to the existence of $R$ with $PSR=P$. 
\end{proof}

\begin{lemma}
Let $T=P(S\rho+T^\prime)$ be a normal boundary condition and $\wt{T}=\wt{P}(\rho+\wt{T}^\prime)$ 
associated with $T$ 
as in $\mathrm{\eqref{eq:canonical}}$. Then 
 $$\scrH^s(\pO,E,\wt{P})=R\big(\scrH^s(\pO,F,P)\big).$$
In particular$:$ The canonical form of a normal, elliptic realization is elliptic. 
\end{lemma}
\begin{proof}
Applying the previous Lemma \ref{lem:surj} with $T^\prime=0$, we obtain 
 \begin{align*}
    \wt{P}\big(\scrH^s(\pO,E)\big)&=RPS\big(\scrH^{s}(\pO,E)\big)=RPS\rho\big(H^{s+d}(\Omega,E)\big)\\
    &=R\big(\scrH^s(\pO,F,P)\big).
 \end{align*}
This shows the first claim and that $\wt{T}=RT:H^s(\Omega,E)\to\scrH^{s-d}(\pO,F,\wt{P})$ surjectively. Thus 
the ellipticity follows from the relation with the Fredholm property of the realization, described after Definition 
\ref{def:ell+norm}. 
\end{proof}

\subsection{Key properties of realizations}\label{sec:real03}

In this section we are going to investigate compositions and adjoints of realizations. First, let us observe that 
normal realizations are always densily defined. In fact, writing $T=P(S\rho+T^\prime)=PS(\rho+\wt{T}^\prime)$, 
we see that the kernel of $T$ on $H^d(\Omega,E)$ contains the kernel of $\rho+\wt{T}^\prime$; this kernel, 
however, is known to be dense in $L_2(\Omega,E)$, cf.\ Lemma 1.6.8 of \cite{Grub}. 

\begin{theorem}\label{thm:comp}
Let $B_j:=(A_{j,+}+G_j)_{T_j}$, $j=0,1$, be two realizations of order $d_j$ subject to APS-type boundary 
conditions $T_j=P_j(S_j\rho+T^\prime_j)$. Moreover, let 
 $$A=A_1A_0,\qquad G=(A_{1,+}+G_1)(A_{0,+}+G_0)-A_+,$$
and define the boundary condition 
 $$T:=\begin{pmatrix}T_1\\ T_0(A_{1,+}+G_1)\end{pmatrix}.$$
Then the following statements are valid$:$
\begin{enumerate}
 \item If $B_0$ is elliptic, then $B_1B_0=(A_++G)_T$. 
 \item If both $B_0$ and $B_1$ are elliptic, then so is $B_1B_0$. 
 \item If both $T_0$ and $T_1$ are normal  $($and $R_j$ denotes the right-inverse of $P_jS_j)$, then 
  the boundary condition $\displaystyle\wt{T}:=\begin{pmatrix}R_1T_1\\ R_0T_0(A_{1,+}+G_1)\end{pmatrix}$
  is normal and equivalent to $T$. 
\end{enumerate}
\end{theorem}
\begin{proof}
The case of trivial projections, $P_0=1$ and $P_1=1$, is Theorem 1.4.6 of \cite{Grub}. For (1) and (2) the same 
proof works also in the general case. Concerning (3), it is clear that $\wt{T}$ is equivalent to $T$, due to the 
injectiveness of  $\mathrm{diag}(R_1,R_0)$. Moreover, 
 $$\displaystyle\wt{T}=\begin{pmatrix}\wt{P}_1&0\\0&\wt{P}_0\end{pmatrix}
     \begin{pmatrix}\rho+\wt{T}^\prime_1\\ (\rho+\wt{T}^\prime_0)(A_{1,+}+G_1)\end{pmatrix}$$
with $\wt{P}_j=R_jP_jS_j$ and $\wt{T}^\prime_j=R_j{T}^\prime_j$. According to Theorem 1.4.6 of \cite{Grub}, 
$(\rho+\wt{T}^\prime_0)(A_{1,+}+G_1)$ is a normal boundary condition 
of the form ${S}\rho+T^\prime$. This yields the normality of $\wt{T}$. 
\end{proof}

Let us now turn to the analysis of adjoints. First recall Green's formula 
$($for details see Section 1.3 of \cite{Grub}, for example$)$: 
If $A\in L^d(2\Omega,2E)$ has the transmission 
property with respect to $\pO$, then there exists a matrix 
 $$\frakA=(\frakA_{jk})_{0\le j,k\le d-1},\qquad \frakA_{jk}\in L^{d-1-j-k}_\cl(\pO,E|_{\pO}),$$
whose components are differential operators $($in particular, $\frakA_{jk}=0$ if $j+k\ge d)$ such that 
\begin{equation}\label{eq:greens-formula}
 (A_+u,v)_\Omega=(u,A^*_+v)_\Omega+(\frakA\rho u,\rho v)_{\pO} 
 \qquad\forall\;u,v\in H^d(\Omega,E);
\end{equation}
here $(\cdot,\cdot)_\Omega$ indicates the inner product of $L^2(\Omega,E)$, while  $(\cdot,\cdot)_\pO$ 
is the inner product of $\mathop{\oplus}\limits_{j=0}^{d-1}L^2(\pO,E|_\pO)$. The skew-diagonal elements 
$\frakA_{j(d-1-j)}$ are induced by endomorphisms in $E|_\pO$, acting like 
$i^d(-1)^{d-1-j}\sigma_\psi^d(A)(x,\nu(x))$ in the fibre 
over $x$. The boundary $\pO$ is called \emph{non-characteristic for $A$} if all these endomorphisms are isomorphisms. 
In this case, $\frakA$ is invertible. 

\begin{theorem}\label{thm:adjoint}
Let $(A_{+}+G)_{T}$ be a realization with $G=K\rho+G^\prime$ and boundary condition $T=P(\rho+T^\prime)$ 
in canonical form  $($recall that any normal realization can be represented in this way$)$. 
Assume that the boundary $\pO$ is non-characteristic for $A$ and define 
\begin{align*}
 G_\ad=-(\frakA T^\prime)^*\rho+G^{\prime *}-(KT^\prime)^*,\qquad 
 T_\ad=P_\ad\big(\rho+(K\frakA^{-1})^*\big), 
\end{align*}
with the so-called adjoint projection 
 $$P_\ad=\big(\frakA(1-P)\frakA^{-1}\big)^*.$$
The following is then true$:$ 
\begin{enumerate}
 \item $\mathrm{dom}((A_{+}+G)_{T}^*)\cap H^d(\Omega,E)=\mathrm{dom}((A_{+}^*+G_\ad)_{T_\ad})$. 
 \item If $(A_{+}+G)_{T}$ is elliptic, its adjoint coincides with $(A_{+}^*+G_\ad)_{T_\ad}$. 
\end{enumerate}
\end{theorem}
\begin{proof}
For convenience set $B:=(A_{+}+G)_{T}$. 

$(1)$ Let $u,v\in H^d(\Omega,E)$. Using Green's formula and writing $\rho u=(\rho+T^\prime)u-T^\prime u$ 
we obtain
\begin{align}\label{eq:adj0}
\begin{split}
 \big((A_++G)u,v\big)_\Omega
 = & \big(u,(A^*_++G^{\prime*})v\big)_\Omega-\big(u,T^{\prime*}(\frakA^*\rho+K^*)v\big)_\Omega+\\
 &+\big((\rho+T^\prime)u,(\frakA^*\rho+K^*)v\big)_\pO. 
\end{split}
\end{align}
Now recall that $v\in\mathrm{dom}(B^*)$ if and only if $u\mapsto\big((A_++G)u,v\big)_\Omega$ is continuous on 
$\mathrm{dom}(B)$ with respect to the $L_2(\Omega,E)$-norm. Since the first two terms on the right-hand side of 
\eqref{eq:adj0} are continuous in this sense, it follows that $v\in\mathrm{dom}(B^*)$ if and only if there 
exists a constant $C\ge 0$ such that 
 $$\big|\big((\rho+T^\prime)u,(\frakA^*\rho+K^*)v\big)_\pO\big|\le C\|u\|_{L_2(\Omega,E)}
     \qquad\forall\;u\in\mathrm{dom}(B).$$
According to Proposition 1.6.5 of \cite{Grub}, for every $u\in H^d(\Omega,E)$ and $\eps>0$ there exists 
an $u_\eps\in H^d(\Omega,E)$ with $\|u_\eps\|_{L_2(\Omega,E)}<\eps$ and 
$(\rho+T^\prime)u_\eps=(\rho+T^\prime)u$. Hence $v\in\mathrm{dom}(B^*)$ if and only if 
 $$\big((\rho+T^\prime)u,(\frakA^*\rho+K^*)v\big)_\pO=0 \qquad\forall\;u\in\mathrm{dom}(B).$$
The surjectivity of $\rho+T^\prime:H^d(\Omega,E)\to\scrH^0(\pO,E)$ implies that 
 $$(\rho+T^\prime)(\mathrm{dom}(B))=\mathrm{ker}\,P=\mathrm{im}\,(1-P).$$ 
We conclude that 
$v\in\mathrm{dom}(B^*)$ if and only if 
 $$\big(\phi,(1-P^*)(\frakA^*\rho+K^*)v\big)_\pO=0 \qquad\forall\;\phi\in\scrH^0(\pO,E).$$
Now the claim immediately follows, since $T_\ad=(\frakA^{-1})^*(1-P^*)(\frakA^*\rho+K^*)$. 

$(2)$ If the realization is elliptic, by Proposition \ref{prop:special_parametrix}, proved below, 
there exists an operator 
$R\in B^{-d,0}(\Omega;E,E)$ such that $R(L^2(\Omega,E))\subset\mathrm{dom}(B)$ and 
$C:=(A_++G)R-1$ is smoothing, i.e., has range in $\scrC^\infty(\Omega,E)$. By general facts on 
the adjoint of compositions, $R^*B^*\subset(BR)^*= ((A_++G)R)^*=1+C^*$. Thus 
the result follows from $(1)$.  
\end{proof}

\begin{proposition}\label{prop:special_parametrix}
Let $(A_{+}+G)_{T}$ be elliptic. Then there exists an operator $R\in B^{-d,0}(\Omega;E,E)$ such that 
\begin{enumerate}
 \item $TR=0$; in particular, $R$ maps $H^d(\Omega,E)$ into the domain of $(A_{+}+G)_{T}$. 
 \item $C_0:=(A_++G)R-1$ is a finite-rank smoothing Green operator of type $0$.  
 \item $C_1=:R(A_++G)-1$ coincides on every space $\{u\in H^{s}(\Omega,E)\mid Tu=0\}$, $s>d-1/2$, 
  with a finite-rank smoothing Green operator of type $d$. 
\end{enumerate}
\end{proposition}
\begin{proof}
It is a well-known fact that there exists a $\Lambda_\Omega\in B^{d,0}(\Omega;E,E)$ 
having inverse $\Lambda^{-1}_\Omega\in B^{-d,0}(\Omega;E,E)$. Employing the notation from 
Definition \ref{def:ell+norm}, let us define 
 $$\scrA_0=\begin{pmatrix}A_0\\T_0\end{pmatrix}
     :=\begin{pmatrix}A_++G\\\Lambda T\end{pmatrix}\Lambda^{-1}_\Omega
     \in \scrB^{0,0}(\Omega;(E,0;1),(E,F_\partial;P_\Lambda)).$$
By assumption, $\scrA_0$ is elliptic. We shall now define various projections; note that they all are smoothing 
Green operators of type $0$, since they are integral operators with smooth kernels. 
Applying Theorem \ref{thm:main_toeplitz} and the results of the Appendix, or referring to Theorem 2.3 of \cite{Schu37}, 
there exists a parametrix 
$\scrA_1=(A_1\;\;K_1)\in\scrB^{0,0}(\Omega;(E,F_\partial;P_\Lambda),(E,0;1))$ of $\scrA_0$ such that 
 $$\scrA_1\scrA_0=1-\pi_0,\qquad \scrA_0\scrA_1=1-\pi_1,$$
with projections of the form 
 $$\pi_0=\sum_{j=1}^{n_0}\big(\cdot,v^0_j\big)_{L_2(\Omega,E)}v_j^0,\qquad 
     \pi_1=\sum_{j=1}^{n_1}\big(\cdot,\scrP_\Lambda^*v^1_j\big)_{L_2(\Omega,E)\oplus L_2(\pO,F_\pO)}v_j^1$$
with functions $\{v_1^0,\ldots v_{n_0}^0\}\subset\scrC^\infty(\Omega,E)$ being an $L_2$-orthogonal basis of 
$V_0:=\mathrm{ker}\,\scrA_0$,  
with $\{v_1^1,\ldots v_{n_1}^1\}\subset\scrC^\infty(\Omega,E)\oplus\scrC^\infty(\pO,F_\pO,P_\Lambda)$ 
being an $L_2$-orthogonal basis of a space $V_1$ that complements  
$\scrA_0(H^s(\Omega,E))$ in $H^s(\Omega,E)\oplus H^s(\pO,F_\pO,P_\Lambda)$ simultaneously for all $s$, 
and with $\scrP_\Lambda:=\mathrm{diag}(1,P_\Lambda)$. 
Note that $(1-\scrP_\Lambda)\pi_1=\pi_1(1-\scrP_\Lambda)=0$. 
If we represent $\pi_1$ in block-matrix form, 
 $$\pi_1=\begin{pmatrix}\pi_{11}&\pi_{12}\\ \pi_{21}&\pi_{22}\end{pmatrix}:
     \begin{matrix}H^s(\Omega,E)\\ \oplus\\ H^s(\pO,F_\pO,P_\Lambda)\end{matrix}
     \lra
     \begin{matrix}H^s(\Omega,E)\\ \oplus\\ H^s(\pO,F_\pO,P_\Lambda)\end{matrix},$$
then 
 $$\pi_{21}u=\sum_{j=1}^{n_1}(u,u_j)_{L_2(\Omega,E)}w_j$$
provided 
$v_j^1=u_j\oplus w_j$ with suitable $u_j\in\scrC^{\infty}(\Omega,E)$ and 
$w_j\in\scrC^\infty(\pO,F_\pO,P_\Lambda)$. 
Now let $U=\mathrm{span}(u_1,\ldots,u_{n_1})$ have the $L_2$-orthonormal basis $\{e_1,\ldots,e_n\}$ and define 
 $$\pi_U=\sum_{j=1}^n(\cdot,e_j)_{L_2(\Omega,E)}e_j.$$
Then, by construction, $\pi_{21}(1-\pi_U)=0$. We now claim that $R:=\Lambda_\Omega^{-1}A_1(1-\pi_U)$ is 
the desired operator. In fact, 
\begin{align*}
 \begin{pmatrix}A_++G\\T\end{pmatrix}R
     &=\begin{pmatrix}1&0\\0&\Lambda^{-1}\end{pmatrix}\scrA_0 A_1(1-\pi_U)\\
     &=\begin{pmatrix}1&0\\0&\Lambda^{-1}\end{pmatrix}\begin{pmatrix}1-\pi_{11}\\-\pi_{21}\end{pmatrix}(1-\pi_U)
     =\begin{pmatrix}(1-\pi_{11})(1-\pi_U)\\ 0\end{pmatrix}
\end{align*}
shows that $TR=0$ and that $C_0$ is a finite-rank smoothing Green operator of type $0$. This shows $(1)$ and $(2)$. Finally, on $H^s(\Omega,E)\cap\mathrm{ker}\,T$, 
\begin{align*}
 C_1&=\Lambda_\Omega^{-1}A_1(1-\pi_U)(A_++G)-1=\Lambda_\Omega^{-1}A_1(1-\pi_U)A_0\Lambda_\Omega-1\\
       &=\Lambda_\Omega^{-1}(A_1A_0-K_1T_0)\Lambda_\Omega-1-
            \Lambda_\Omega^{-1}A_1\pi_UA_0\Lambda_\Omega\\
       &=\Lambda_\Omega^{-1}(\scrA_1\scrA_0-1)\Lambda_\Omega-
            \Lambda_\Omega^{-1}A_1\pi_UA_0\Lambda_\Omega\\
       &=-\Lambda_\Omega^{-1}(\pi_0+A_1\pi_UA_0)\Lambda_\Omega,  
\end{align*}
proving claim $(3)$. 
\end{proof}

\begin{corollary}\label{cor:closed}
If $(A_{+}+G)_{T}$ is elliptic, it is a closed operator in $L_2(\Omega,E)$.
\end{corollary}
\begin{proof}
Let $(u_n)$ be a sequence in $\mathrm{dom}((A_{+}+G)_{T})$ such that both $u:=\lim\limits_{n\to+\infty}u_n$ and 
$v:=\lim\limits_{n\to+\infty}Au_n$ exist in $L_2(\Omega,E)$. 

Let $R\in B^{-d,0}(\Omega;E,E)$ be the parametrix constructed in Proposition \ref{prop:special_parametrix} and 
$C_1$ the respective smoothing Green operator. Then 
$C_1u_n=R(A_{+}+G)u_n-u_n$ is convergent in $L_2(\Omega,E)$. Since $C_1$ maps the domain of 
$(A_{+}+G)_{T}$ into a finite-dimensional subspace of $\scrC^\infty(\Omega,E)$, the sequence $(C_1u_n)$ is 
also convergent in $H^d(\Omega,E)$. Thus $u_n=R(A_{+}+G)u_n-C_1u_n$ converges in $H^d(\Omega,E)$. 
We conclude that $u\in H^d(\Omega,E)$, $v=(A_{+}+G)u$, and $Tu= \lim\limits_{n\to+\infty}Tu_n=0$. 
\end{proof}

\subsection{Self-adjoint realizations}\label{sec:real04}

A realization may be represented in many different ways. In the present section we analize this fact 
systematically and then characterize the self-adjoint realizations. 

\begin{proposition}\label{prop:equiv_cond}
Let $T_j=P_j(\rho+T^\prime_j)$, $j=0,1$, be two boundary conditions in normal form. 
Then $T_0$ and $T_1$ are equivalent if, and only if, $P_j(1-P_{1-j})=0$ and 
$P_jT_j^\prime=P_jT_{1-j}^\prime$ for $j=0,1$. 
\end{proposition}

Note that the property $P_j(1-P_{1-j})=0$ for $j=0,1$ is equivalent to $\mathrm{ker}\,P_0=\mathrm{ker}\,P_1$ 
for $P_0$ and $P_1$ considered as maps on $\scrH^s(\pO,E)$  for some $($and then every$)$ choice of $s$. 
Then $P_0T_0^\prime=P_0T_1^\prime$ is equivalent to $P_1T_1^\prime=P_1T_0^\prime$. 

\begin{proof}[Proof of Proposition $\mathrm{\ref{prop:equiv_cond}}$]
Recall that the boundary conditions are called equivalent if their kernels on $H^d(\Omega,E)$ coincide.

First let us show that the stated conditions imply the equivalence. 
Clearly $T_0u=0$ means $(\rho+T^\prime_0)u\in\mathrm{ker}\,P_0$. Thus, by assumption, also 
$0=P_1(\rho+T^\prime_0)u=P_1(\rho+T^\prime_1)u=T_1u$. Interchanging roles of $T_0$ and $T_1$ 
thus shows $\mathrm{ker}\,T_0=\mathrm{ker}\,T_1$. 

Now let us assume that the kernels coincide. According to Lemma 1.6.8 of \cite{Grub} there exists a 
right-inverse $K$ to $\rho$ such that $\Lambda:=1+KT_0^\prime$ is an isomorphism in $H^s(\Omega,E)$ 
simultaneously for all $s\ge0$. Note that $P_0\rho\Lambda=T_0$. Thus, for $u\in H^d(\Omega,E)$, 
\begin{align}\label{eq:xxx}
\begin{split}
 P_0\rho\Lambda u=0 & \iff T_1u=0 \\ 
 & \iff P_1\rho \Lambda u + P_1\big(\rho(1-\Lambda)+T_1^\prime\big)u =0\\
 & \iff P_1\rho \Lambda u + P_1(T_1^\prime-T_0^\prime)u=0.
\end{split}
\end{align}
This equivalence implies, in particular, that 
 $$P_1(T_1^\prime-T_0^\prime)u=0\qquad
    \forall\;u\in U:=\Lambda^{-1}\big(\scrC^\infty_0(\mathrm{int}\,\Omega,E)\big).$$
Since $U$ is dense in $L_2(\Omega,E)$ and $T_j^\prime$ is of type $0$, it follows that 
$P_1(T_1^\prime-T_0^\prime)=0$, i.e., $P_1T_0^\prime=P_1T_1^\prime$. Then  
\eqref{eq:xxx} and the surjectivity of $\rho\Lambda:H^d(\Omega,E)\to \scrH^{0}(\pO,E)$ show 
that $P_0$ and $P_1$ have the same kernel on $\scrH^{0}(\pO,E)$. Interchanging roles of 
$T_0$ and $T_1$ yields also $P_0T_1^\prime=P_0T_0^\prime$. 
\end{proof}

Let $B=(A_++G)_T$ with $T=P(\rho+T^\prime)$. One can always choose $G=K\rho+G^\prime$ in a certain 
\emph{reduced form}, namely with $K$ satisfying $KP=0$. In fact, if initially $G=K_0\rho+G_0^\prime$ and $Tu=0$ 
$($i.e., $P\rho u=-T^\prime u)$, we can write 
 $$Gu=K_0(P\rho u+(1-P)\rho u)+G_0^\prime u=K_0(1-P_0)\rho u+ (G_0^\prime-K_0PT^\prime)u$$ 
and then set $K:=K_0(1-P)$ and $G^\prime:=G_0^\prime-K_0PT^\prime$. 

\begin{proposition}\label{prop:equiv_action}
With $j=0,1$ let $B_j=(A_++G_j)_{T}$ be two realizations with $T=P(\rho+T^\prime)$ and 
$G_j=K_j\rho+G_j^\prime$ in reduced form, i.e., $K_jP=0$. Then $B_0=B_1$ if, and only if, 
$K_0=K_1$ and $G_0^\prime=G_1^\prime$. 
\end{proposition}
\begin{proof}
Clearly $B_0=B_1$ if, and only if,  
 $$(A_++G_0)u=(A_++G_1)u\qquad \forall\; u\in H^d(\Omega,E)\cap\mathrm{ker}\,T.$$
If $\Lambda$ is an isomorphism associated with $T$ as in the proof of Proposition \ref{prop:equiv_cond}, 
this is equivalent to 
 $$(G_0-G_1)\Lambda^{-1}v=0\qquad \forall\; v\in H^d(\Omega,E)\cap\mathrm{ker}\,P\rho.$$
For such $v$ we can write 
 $$(G_0-G_1)\Lambda^{-1}v=(K_0-K_1)\rho v+Gv$$ 
with $G:=(K_0-K_1)\rho(\Lambda^{-1}-1)+(G_0^\prime-G_1^\prime)\Lambda^{-1}$ having type 0, according 
to Lemma 1.6.8 of \cite{Grub}. Choosing $v\in\scrC^\infty_0(\mathrm{int}\,\Omega,E)$ we derive that 
$G=0$ and that 
 $$(K_0-K_1)\rho v=0\qquad \forall\; v\in H^d(\Omega,E)\cap\mathrm{ker}\,P\rho.$$
since 
$\rho:H^d(\Omega,E)\to\scrH^0(\pO,E)$ surjectively, this means 
 $$(K_0-K_1)\phi=0\qquad \forall\;\phi\in \scrH^0(\pO,E)\cap\mathrm{ker}\,P.$$
Since $\mathrm{ker}\,P=\mathrm{im}\,(1-P)$ we derive that $(K_0-K_1)(1-P)=0$ and thus 
$K_0-K_1=0$, since $(K_0-K_1)P=0$ by assumption. From $G=0$ we then obtain $G_0^\prime=G_1^\prime$. 
\end{proof}

As a consequence we obtain the following description of self-adjointness for realizations: 

\begin{theorem}\label{thm:self-adjoint}
Consider an elliptic realization $B=(A_++G)_T$ with $A$ being symmetric, $T=P(\rho+T^\prime)$ and 
$G=K\rho+G^\prime$ in reduced form. Assume that $\pO$ is non-characteristic for $A$. 
Then$:$
\begin{enumerate}
 \item $\mathrm{dom}(B^*)=\mathrm{dom}(B)$ if, and only if, $\frakA:\mathrm{ker}\,P\to(\mathrm{ker}\,P)^\perp$ 
  isomorphically and $P(T^\prime+\frakA^{-1}K^*)=0$. 
 \item If $\mathrm{dom}(B^*)=\mathrm{dom}(B)$ then $B=B^*$ if, and only  if, 
  \mbox{$G^\prime=G^{\prime *}-(KT^\prime)^*-T^{\prime *}\frakA T^\prime$}. 
\end{enumerate}
Let us note that one always may assume that $T^\prime=PT^\prime$ in the representation of $T$. In this case, 
the term $(KT^\prime)^*$ in $(2)$ vanishes, since $KP=0$ by assumption.  
\end{theorem}
\begin{proof}
We have $B^*=(A_++G_\ad)_{T_\ad}$ as described Theorem \ref{thm:adjoint}. The symmetry of $A$ implies 
that $\frakA^*=-\frakA$ and therefore 
\begin{align*}
 G_\ad&=T^{\prime*}\frakA\rho+G^{\prime *}-(KT^\prime)^*,\\
 T_\ad&=P_\ad\big(\rho-\frakA^{-1}K^*\big),\\ 
 P_\ad&=\frakA^{-1}(1-P^*)\frakA. 
\end{align*}
Hence, by Proposition \ref{prop:equiv_cond} and the comment given thereafter, 
the domains of $B$ and $B^*$ coincide if, and only if, 
$\mathrm{ker}\,P_\ad=\mathrm{ker}\,P$ and $PT^\prime=-P\frakA^{-1}K^*$. Now $(1)$ follows, since 
 $$u\in\mathrm{ker}\,P_\ad \iff 
     \frakA u\in \mathrm{ker}\,(1-P^*)=\mathrm{im}\,P^*=(\mathrm{ker}\,P)^\perp.$$ 
Let us now show $(2)$. We have $B^*=(A_++G_\ad)_T$, since $T$ and $T_\ad$ are equivalent by assumption. 
Writing $\rho=(1-P)\rho+P\rho$ and using the fact that $P\rho u=-T^\prime u$ provided $Tu=0$, 
the reduced form of $G_\ad$ is 
 $$G_\ad=T^{\prime*}\frakA(1-P)\rho+G^{\prime *}-(KT^\prime)^*-T^{\prime *}\frakA T^\prime.$$
According to Proposition \ref{prop:equiv_action}, $B=B^*$ is equivalent to 
 $$K=T^{\prime*}\frakA(1-P)\quad\text{and}\quad
     G^\prime=G^{\prime *}-(KT^\prime)^*-T^{\prime *}\frakA T^\prime.$$
Since $KP=0$ by assumption, 
\begin{align*}
 K=T^{\prime*}\frakA(1-P) 
 &\iff (K-T^{\prime*}\frakA)(1-P)=0\\
 &\iff (1-P^*)(K^*+\frakA T^{\prime})=0 \\ 
 &\iff P_\ad (T^\prime+\frakA^{-1}K^*)=0
\end{align*}
However, this is true by $(1)$ $(P_\ad$ can be equivalently replaced by $P$, since $P_\ad$ and $P$ have the 
same kernel$)$. 
\end{proof}

Theorem \ref{thm:self-adjoint} in case of $B=(A_+)_T$ with symmetric $A$ and $T=P\rho$, states that the 
self-adjointness of $B$ is equivalent to $\frakA:\mathrm{ker}\,P\to(\mathrm{ker}\,P)^\perp$ being an isomorphism. 

\section{Spectral triples for manifolds with boundary}
\label{sec:spectral}

A triple $(\mathcal{A},\mathcal{H},\mathcal{D})$ is called a spectral triple of dimension $n\in\nz$ if 
\begin{itemize}
 \item[a$)$] $\calH$ is a Hilbert space and $\mathcal{A}$ is a unital, involutive algebra, 
  faithfully represented in $\scrL(H)$,  
 \item[b$)$] $\mathcal{D}$ is a closed, self-adjoint operator in $\mathcal{H}$ with compact resolvent and such that 
  the sequence of eigenvalues $\mu_1\le\mu_2\le\cdots$ of $|\calD|$ satisfies $\mu_j\sim j^{1/n}$.  
 \item[c$)$] for every $a\in\mathcal{A}$, application of $a$ preserves $\domain(\mathcal{D})$ and the 
  commutator $[\mathcal{D},a]$, initially defined on $\domain(\mathcal{D})$, extends by continuity to an operator 
  in $\scrL(H)$, denoted by $da$ $($thus $da=\overline{[\mathcal{D},a]})$.  
\end{itemize}
To define the notion of regular spectral triple, we shall need the operator 
 $$\delta:\domain(\delta)\lra\scrL(\calH),\qquad L\mapsto\delta(L):=\overline{[|\calD|,L]},$$
whose domain consists of those operators $L\in\scrL(\calH)$ that map $\domain(\calD)$ into itself and 
whose commutator $[|\calD|,L]$ extends by continuity to a bounded operator in $\calH$.  

\begin{definition}
A spectral triple $(\mathcal{A},\mathcal{H},\mathcal{D})$ is called regular if, for every $a\in\mathcal{A}$,  
 $$a,da\in\domain(\delta^k)\qquad\forall\;k\in\nz.$$ 
\end{definition}

In the sequel we shall focus on the case that $\mathcal{H}:=L^2(\Omega,E)$ with $n=\mathrm{dim}\,\Omega$ 
and that  $\mathcal{A}\subseteq\scrC^\infty(\Omega)$, represented in $\scrL(\calH)$ as operators of multiplication 
with functions. We now shall analyze when a first order, elliptic, self-adjoint realization $\calD=(A_++G)_T$ subject to 
APS-type conditions leads to a spectral triple of dimension $n$. 

First of all we note that if $\calA^0_ \calD$ is defined as 
\begin{equation}\label{eq:A0D}
 \calA^0_\calD:=\{a\in\scrC^\infty(\Omega)\mid \text{both $a$ and $a^*$ map $\domain(\calD)$ into itself}\},
\end{equation}
then $(\mathcal{A}^0_\calD,\mathcal{H},\mathcal{D})$ is a spectral triple provided $G$ is a Green operator of order 
and type 0. In fact, for any $a\in\calA^0_\calD$, 
 $$[A_++G,a]=[A,a]_++[G,a]\in\scrL(\calH),$$
since $[A,a]$ is a pseudodifferential operator of order 0 and $[G,a]$ is Green operator of order and type 0.
Moreover, by self-adjointness, $(\calD-i)^n$ induces a bijection $\domain(\calD^n)\to\calH$. 
The fact that $\domain(\calD^n)\subset H^n(\Omega,E)$ together with a general, functional-analytic result 
$($see e.g. Lemma A.4 in \cite{Grub}) now implies that $\mu_j(|\calD|)\sim j^{1/n}$. 

The situation for regular spectral triples is more complicated$:$ 

\begin{theorem}\label{th:spectral}
Let $\mathcal{H}=L^2(\Omega,E)$ and $\calD:=(A_+)_T$ be an elliptic, 
self-adjoint\footnote{Recall that $\calD$ is self-adjoint if, and only if, 
$A$ is symmetric and $\mathrm{ker}\,T=\mathrm{ker}\,T_\ad$.} realization of first order with boundary 
condition of the form $Tu=P(Su|_{\pO}+T^\prime u)$ $($cf.\ Definition $\mathrm{\ref{def:aps}}$ with 
$d=1)$. Assume  that $A^2$ has scalar principal symbol and that 
\begin{equation}\label{eq:+}
  A_+P_+=(AP)_+ \quad\forall\;
 \text{non-negative order pseudodifferential operators $P$}.\footnotemark
\end{equation}
Let\footnotetext{For example this is the case if $A=A_0+A_1$ where $A_0$ is a differential operator and 
$A_1$ is pseudodifferential with $A_1=A_1\varphi$ with a smooth function $\varphi$ supported in the interior of $\Omega$. 
Also more general $A_1$ are possible, but we shall not enter details here.} 
$\calA^\infty_\calD$ be defined as
   $$\calA^\infty_\calD
       :=\{a\in\calA^0_\calD \mid \text{both $a$ and $a^*$ map $\calH^\infty_\calD$ into itself }\},$$
where $\calA^0_\calD$ is as in \eqref{eq:A0D} and 
\begin{equation*}
 \calH^\infty_\calD=\mathop{\mbox{\Large$\cap$}}_{k\in\nz}\domain(\calD^k)
\end{equation*}
$($note that in the definition of $\calH^\infty_\calD$ one may also use the operator $|\calD|$ in place of $\calD)$. 
Then $(\mathcal{A}^\infty_\calD,\mathcal{H},\mathcal{D})$ is a regular spectral triple. 
Moreover, $\calA^\infty_\calD$ is the largest subalgebra of $\calA^0_\calD$ that, 
together with $\calD$, leads to a regular spectral triple.
\end{theorem}

\begin{proof}
The proof is along the lines of that of Theorem 4.5 in \cite{IoLe}. 
Clearly, $(\mathcal{A}^\infty_\calD,\mathcal{H},\mathcal{D})$ is a spectral triple of dimension $n$, since 
$\calA^\infty_\calD$ is a $*$-subalgebra of $\calA^0_\calD$. 
Now observe that, by construction, if $b=a$ or 
$b=[\calD,a]$ with $a\in\calA^\infty_\calD$, then $b$ maps $\calH^\infty_\calD$ into itself.
Thus we can define the iterated commutators
 $$b^{(k)}:=[\calD^2,\cdot]^k(b):\calH^\infty_\calD\lra\calH^\infty_\calD,\qquad k\in\nz.$$
By Lemma 2.6 of \cite{IoLe}, to 
prove regularity of the spectral triple, it suffices to show that 
\begin{equation}\label{eq:k}
 \|b^{(k)}u\|_{L^2(\Omega,E)}\le C_k\|u\|_{H^k(\Omega,E)}\qquad\forall\;u\in \calH^\infty_\calD,\quad\forall\;k\in\nz,
\end{equation}
with constants $C_k$ not depending on $u$. 

Due to assumption \eqref{eq:+}, $(A_+)^\ell=(A^\ell)_+$ for every 
$\ell\in\nz$. Therefore, in case $b=a$, property \eqref{eq:k} immediately follows, since  
 $$b^{(k)}=([A^2,\cdot]^{(k)}(a))_+$$
is $($the restriction to $\Omega$ of$)$ a pseudodifferential operator of order $k$, since $A^2$ has scalar 
principal symbol. 

To verify \eqref{eq:k} in case of $b=[\calD,a]=[A_+,a]=[A,a]_+$, first observe that 
with $A$ also $b$ satisfies condition \eqref{eq:+}, since 
\begin{align*}
 [A_+,a]P_+&= A_+aP_+-aA_+P_+=A_+(aP)_+-a(AP)_+\\
  &=(AaP)_+-(aAP)_+=([A,a]P)_+;
\end{align*}
here we have used that $a$ is zero order differential. Then $b^{(k)}=([A^2_+,\cdot]^k(b)$ is a 
linear combination of terms $A^{2k_1}_+bA^{2k_2}_+$ with $k_1+k_2=k$. Applying repeatedly 
Property \eqref{eq:+}, any such term equals $(A^{2k_1}[A,a]A^{2k_2})_+$. We conclude that 
 $$([A^2_+,\cdot]^k(b)=\big([A^2,\cdot]^k([A,a])\big)_+,$$
and then can argue again as before to obtain the mapping property \eqref{eq:k}. 

For the last affirmation of the theorem, assume that 
$(\mathcal{A},\mathcal{H},\mathcal{D})$ is a regular spectral triple with $\calA$ 
being a $*$-subalgebra of $\calA^0_\calD$. Then, for $b=a$ or $b=a^*$ with $a\in\calA$, the following identity 
holds $($see Lemma 2.1 in \cite{Co13})$:$
\begin{align*}
 |\calD|^n (bu)=\sum_{j=0}^n\binom{n}{j} \delta^j(b)  |\calD|^{n-j}u \qquad\forall\;u\in\domain(|\calD|^n).
\end{align*}
This shows at once that $bu\in\calH^\infty_\calD$ provided $u\in\calH^\infty_\calD$. 
Thus $\calA\subseteq \calA^\infty_\calD$. 
\end{proof}

The precise description of $\calA^\infty_\calD$ is in general very cumbersome and even in specific examples 
it appears very difficult to provide an explicit expression of this algebra. However, the following is valid$:$

\begin{lemma}\label{thm:1carA}
Let $\calD=(A_+)_T$ be as described in Theorem \textnormal{\ref{th:spectral}} with $A$ differential    
and assume that the boundary condition is of the form $Tu=P S u|_{\partial\Omega}$ with a projection $P$ and 
a bundle homomorphsm $S$ on the boundary. 
Let $a\in\scrC^\infty(\Omega)$ and assume that there exists a 
smooth function $\varphi$ which is constant near every connected coponent of the boundary of $\Omega$ 
such that $a-\varphi$ vanishes to infinite order at $\pO$. Then $a\in\calA^\infty_\calD$. 
\end{lemma} 
\begin{proof}
Since $\varphi$ belongs to $\calA^\infty_\calD$, we may assume that $\varphi\equiv0$. 
If $u\in\calH^\infty_\calD$ then $au$ also vanishes to infinite order at 
the boundary, hence so does $A_+^N(au)$ for arbitrary $N\in\nz$. Therefore 
 $$T(A_+^N(au))=P S (A_+^N(au)|_{\partial\Omega})=0,$$
showing that $au\in\domain(\calD^{N+1})$. For the same reason, $a^*u\in\domain(\calD^{N+1})$. 
Since $N$ is arbitrary it follows that both $a$ and $a^*$ preserve $\calH^\infty_\calD$. 
\end{proof}

As we shall see in the sequel, cf.\ Theorem \ref{thm:closure} below, it seems easier to describe the closure 
in $\scrL(\calH)$ of $\calA^\infty_\calD$ $($or, if we consider $\calA^\infty_\calD$ as a subspace of the 
continuous functions on $\Omega$, its closure with respect to the supremum-norm$)$. 
As already said, this is of significance in view of Connes' reconstruction Theorem; in fact, 
in case a spectral triple $\pt{\calA, \calH,\calD}$ satisfies Connes' axioms, the reconstructed manifold 
is homeomorphic, as a topological space, to $\mathrm{Spec}(\overline{\calA})$.
We finish this subsection with a technical lemma which we shall employ below in this context. 

\begin{lemma}\label{lem:commute}
Let $\pt{A_+}_T$ be as described in Theorem \textnormal{\ref{th:spectral}} with $A$ differential 
and assume that the 
boundary condition is of the form $Tu=P S u|_{\partial\Omega}$ with an orthogonal projection $P$ and 
a bundle isomorphism $S$ on the boundary. Let $a\in\scrC^\infty(\Omega)$ and assume that both $a$ and 
$a^*$ preserve $\mathrm{dom}\pt{\pt{A_+}_T}$. Then $a|_{\partial\Omega}$ commutes with $P$. 
\end{lemma}
\begin{proof}
Let $\varphi\in\scrC^\infty(\partial\Omega,E|_{\partial\Omega})$ arbitrary and $u$ be some function in 
$\scrC^\infty(\Omega,E)$ such that $u|_{\partial\Omega}=S^{-1}(1-P)\varphi$. 
Then $u\in\mathrm{Dom}(\calD)$ and hence, by assumption, 
  $$0=T(au)=P (\wt{a} (1-P)\varphi),\qquad \wt{a}:=a|_{\partial\Omega}.$$
Thus $P \wt{a} (1-P)=0$, i.e., $P \wt{a}=P \wt{a} P$. 
Replacing $a$ by $a^*$ shows that $P \wt{a}^*=P \wt{a}^* P$. Passing to 
adjoints yields $\wt{a} P=P \wt{a} P=P \wt{a}$. 
\end{proof}

\subsection{Self-adjoint realizations of Dirac operators}\label{sec:4.1} 

  In order to define a Dirac operator we suppose that $(\Omega, E)$ is a Clifford module 
  and that the bundle $E$ has an Hermitian structure $\langle\cdot,\cdot\rangle$ and a connection
  $\nabla$ compatible with the
  Clifford module structure.
  We call $D$ the associated Dirac operator; it is symmetric and locally has the form
  \[
    \pt{D u}(x)= \sum_{j=1}^n c\pt{e_j} \pt{\nabla_{e_j} u }(x), \quad u\in \scrC^{\infty}\pt{\Omega, E}
  \]
   where $\{e_1, \ldots, e_n\}$ is an orthonormal frame of $T\Omega$ at $x\in \Omega$
   and $c(\cdot)$ is the Clifford multiplication.

Depending on the parity of the dimension of $\Omega$, we can complete $D$ with 
APS-type boundary conditions to an elliptic, self-adjoint realization. 
  
\textbf{The case of even dimension:}
      In this case the bundle $E$ canonically splits in two subbundles $E_+$ and $E_-$
      via the chirality operator, i.e., $E=E_+\oplus E_-$,  
      and $D: \scrC^{\infty}(\Omega, E)\to \scrC^{\infty}(\Omega, E)$ can be identified with 
\begin{align}
\begin{split}
 D=\begin{pmatrix}0&D^-\\ D^+&0\end{pmatrix},\qquad 
 D^\pm: \scrC^{\infty}(\Omega, E_\pm) \to \scrC^{\infty}(\Omega, E_\mp),
\end{split}
\end{align}
where $\pt{D^+}^*= D^-$. 
Recall that, in a collar neighborhood of the boundary $\pO$, the metric is assumed to be of product type  
$($in case of Dirac operators this is not a restrictive assumption as it can be always achieved up to conjugation by 
unitary isomorphism, see the appendix of \cite{BLZ09}, for instance$)$. 
Then one can write, near the boundary, 
      \begin{align}
        \label{eq:dirprod}
        D= \Gamma(x^\prime) \pt{\partial_{x_n}+ B}
      \end{align}
where $(x^\prime, x_n)$ are the normal and tangential coordinates, respectively, and an endomorphism 
$\Gamma:E_+ \oplus E_- \to E_- \oplus E_+$ which inverts the chirality
and does not depend on the normal direction. In fact, it corresponds to the Clifford multiplication with the inward 
normal vector; in particular, $\Gamma^2=- \mathrm{Id}_E$. 
The so-called tangential operator 
      \begin{align*}
      B: \scrC^{\infty}(\pO, E_+|_{\pO}\oplus E_-|_{\pO}) \to \scrC^{\infty}(\pO, E_+|_{\pO}\oplus E_-|_{\pO})
      \end{align*}
is a self-adjoint elliptic differential operator of first order preserving the splitting $E=E_+\oplus E_-$. 
Since $B$ is elliptic and self-adjoint, there are well defined eigenvalues
$\{\lambda_k\}_{k\in \gz}$ and eigenfunctions $\{f_k\}_{k \in \gz}$ which form an orthonormal
base of $L^2(\pO, E)$. We consider 
      \begin{align*}
        P_{\geq}: L^2\pt{ \pO, E } \lra  L^2\pt{\pO, E},\qquad 
        u \mapsto \sum_{\lambda_k \geq 0} \spk{u,f_k}f_k,         
      \end{align*}
that is the orthogonal projection onto the span of the eigenfunctions corresponding to non-negative eigenvalues. 
We set $P_<=1-P_{\ge}$.  The APS boundary condition is then defined as
 $$T_\aps= \begin{pmatrix} 
                     P_{\geq} \gamma_0&0 \\  0&P_{<} \Gamma^*\gamma_0 
                    \end{pmatrix}
     :H^s(\Omega, E_+\oplus E_-)\lra H^{s-1/2}(\pO,E_+\oplus E_-).$$
Then we let $\calD=\calD_\aps$ denote the realization of $D$ with subject to the condition $T_\aps$. 
The splitting $E=E_+\oplus E_-$ induces an identification 
$\calD_\aps=\begin{pmatrix}0&D^-_\aps\\ D^+_\aps&0\end{pmatrix}$ 
with the operators $D^\pm_\aps$ given by the action of $D^\pm$ on the domains
\begin{align}\label{eq:diracaps}
\begin{split}
 \mathrm{dom}(D^+_\aps)&=\big\{ u \in H^1(\Omega, E_+)\mid P_{\geq}u|_{\pO}=0\big\}, \\
 \mathrm{dom}(D^-_\aps)&=\big\{ u \in H^1(\Omega, E_-)\mid P_{<}(\Gamma^*u|_{\pO})=0\big\}. 
\end{split}
\end{align} 
Theorem \ref{thm:self-adjoint} implies that $\calD_\aps$ is self-adjoint. 
Moreover, by Theorem \ref{th:spectral}, both  
$(\calA^{0}_{\calD},\calH, \calD)$ and 
$(\calA^{\infty}_{\calD},\calH, \calD)$ with $\calH=L^2(\Omega,E)$ and $\calD=\calD_{\aps}$  
are spectral triples of dimension $n$, the second one being regular. 

\begin{remark} \label{rem:evenspectral}
If $\Omega$ were without boundary, it is known that $(\scrC^{\infty}(\Omega), L^2(\Omega, E), \calD)$
defines a regular spectral triple. Moreover, in the even dimensional case, it is possible to define a grading
$\gamma$ of $L^2(\Omega, E)$ such that the spectral triple is even, i.e., 
$\gamma^2=\mathrm{Id}$, $\gamma^*=\gamma$, $\gamma a=a \gamma$ for all $a \in \scrC^{\infty}(\Omega)$ and 
$\gamma \calD +\calD \gamma=0$. The grading is related to the splitting induced by the chirality.
 
The spectral triples $(\calA^{0}_{\calD},\calH, \calD)$ and $(\calA^{\infty}_{\calD},\calH, \calD)$
introduced above are also even, since the grading operator preserves the domain of $\calD_{\aps}$. 
This is not true for the spectral triples introduced in \cite{IL13}, based on chiral boundary conditions. 
Dealing with local boundary conditions, as chiral boundary conditions, the grading does not preserve the domain of
the Dirac operator, see e.g. the example on the disk in \cite{CC13}. 
\end{remark}

\textbf{The case of odd dimension:}
We suppose again that the metric is of product type near the boundary and thus $D$ has the form 
\eqref{eq:dirprod}. Also in this case it is well known that the APS boundary conditions define  
elliptic realizations of the Dirac operator. Since $D$ is essentially self-adjoint, 
\begin{align}
  \label{eq:anticom}
  \Gamma B=B \Gamma^*= - B \Gamma,
\end{align}
i.e., $\Gamma$ inverts the splitting induced by the spectrum of $B$.
For Green's formula \eqref{eq:greens-formula} we find $\mathfrak{A}=\Gamma^*=-\Gamma$.
By \eqref{eq:anticom}, we obtain
\[
  \mathfrak{A}=-\Gamma: \mathrm{ker}\,P_{\geq}\lra \mathrm{ker}\,P_{\leq}. 
\]
In case $B$ is invertible, 
$\mathrm{ker}\,P_{\leq}=\mathrm{ker}\,P_{<}= (\mathrm{ker}\,P_{\geq})^\perp$. 
Hence, by Theorem \ref{thm:self-adjoint}, the realization $\calD=\calD_{\aps}$ of $D$ 
subject to the boundary condition $T=P_{\geq 0} \gamma_0$ is self-adjoint. 

In case $B$ is not invertible, the usual $\aps$-boundary condition does not give a 
self-adjoint realization and one has to proceed differently, as is discussed in detail in \cite{DF94}. 
Let $a^2 \in \rz \setminus \sigma (B^2)$ and  let 
 $E(\lambda) \subset L^{2}(\partial \Omega, E|_{\partial \Omega})$,  
denote the eigenspace associated to the eigenvalue $\lambda \in \sigma (B)$.
Since the boundary is even dimensional, there is a splitting $E|_{\pO}=E|_{\pO}^+\oplus E|_{\pO}^-$,
inducing a splitting of each eigenspace. We set
 \[
   K^{\pm}_{\pO}(a)= \mathop{\mbox{\Large$\oplus$}}_{-a<\lambda<a}E^{\pm}(\lambda)
 \]
and let $P_{g}\in\scrL( L^{2}\pt{\pO, E|_{\pO}})$ be the orthogonal projection onto the graph of 
an $L_2$-unitary map $g: K^{+}_{\pO}(a)\to K^{-}_{\pO}(a)$. Then the trace operator
 \[
   T=\pt{P_{>a}+P_{g}}\gamma_0
 \]
is an APS-type boundary condition as defined in Section \ref{sec:real} 
$($note that $P_{>a}P_{g}=P_gP_{>a}=0$, hence $P_{>a}+P_g$ indeed is a projection$)$
and induces an elliptic, self-adjoint realization of the Dirac operator, again denoted by 
$\calD$ $($of course this operator depends on the choice of $a$ and $g)$. 

In any case, by Theorem \ref{th:spectral}, we can conclude that both 
$(\calA^{0}_{\calD},L^2\pt{\Omega, E} , \calD)$ and 
$(\calA^{\infty}_{\calD},L^2\pt{\Omega, E} , \calD)$ are
 spectral triples of dimension $n$, the second one being regular.

\begin{remark} \label{rem:genproj}
In \cite{Sc95}, the realizations of the Dirac operator subject to 
$\aps$-type conditions $T=P\gamma_0$ are analyzed, where $P$ is a zero order pseudodifferential projection on the boundary.
In case of odd dimension and invertible tangential operator, it is proven that such a realization is elliptic and 
self-adjoint if, and only if, $P=P_g$ is the orthogonal projection onto the graph of an $L^2$-unitary isomorphism 
$g: E^+_{\pO}\to E^-_{\pO}$. 
The case of a non-invertible tangential operator is studied in \cite{DF94}. 
It is proven that $T=P\gamma_0$ leads to a elliptic and self-adjoint extension of $D_{P_{>-a}} \gamma_0$ 
if and only if $P=P_{>a}+P_g$ with $P_g$ described above. 
  
We want to stress once more that all these boundary conditions are of APS-type and fit in our general framework.
\end{remark}

\subsection{Spectral triples based on Dirac operators} 

We shall now study $\calA^0_{\calD}$ and $\calA^\infty_{\calD}$ in case of $\calD$ being an above 
described realization of the Dirac operator. 

\begin{proposition}\label{prop:const}
Let $(\calA^0_{\calD},\calH,\calD)$  be the spectral triple associated with a Dirac operator 
$D$ and $\aps$-type boundary conditions as described in the previous Subsection $\mathrm{\ref{sec:4.1}}$. 
Let $a\in \calA^0_{\calD}$ such that both $a$ and $a^*$ preserve $\mathrm{Dom}(\calD)$. 
Then $a|_{\partial\Omega}$ is locally constant.  
\end{proposition}
\begin{proof}
In case  $\Omega$ is of even dimension, both $a$ and $a^*$ preserve the domain of 
$D^+_{\aps}$, cf.\ \eqref{eq:diracaps}. In the case of odd dimension let us first assume that the tangential operator 
is invertible. In both cases, the boundary condition is $T=P_{\geq}\gamma_0$ and Lemma \ref{lem:commute} 
implies that both $a$ and $a^*$ commute with $P_{\geq}$.

It is well known, see \cite[\S 14]{BW92}, that $P_{\geq}$ is a classical pseudodifferential
operator of order zero, having principal symbol
 \begin{align} 
    p_0(x^\prime, \xi^\prime)= \frac{b^\prime(x^\prime, \xi^\prime)+1}{2},
 \end{align}
where $b^\prime(x^\prime, \xi^\prime)$ is the principal symbol of the operator $(1+B^2)^{-1/2}B$.
In the even dimensional case, $B$ can be identified with a Dirac operator on the boundary, therefore
\begin{align*}
  b^\prime(x^\prime, \xi^\prime)= \frac{i \xi^\prime\cdot} {|\xi^\prime|}, \quad 
  \xi^\prime\cdot \text{ being the Clifford multiplication on }\partial \Omega.
\end{align*}
In particular, $b'(x^\prime, \xi^\prime)$ is not constant as a function of 
$\xi^\prime$.
In the odd dimensional case it is also possible to give the explicit expression of the principal symbol of
$B$ and state that it is not constant, see e.g. \cite{DFerra}.\footnote{Here, we are supposing the dimension
to be at least two. The one dimensional case is trivial: clearly the restriction to the boundary is constant
since it is the evaluation at one point.}
For notational convenience, let us now simply write $a$ instead of $a|_{\partial\Omega}$.  
Since $a P_{\geq}= P_{\geq} a$, 
in particular, the local symbols of $a P_{\geq}$ and $P_{\geq} a$ are equal.
Let us suppose that the symbol of $P_{\geq}$ has the asymptotic expansion into homogeneous components  
$\sum\limits_{j=0}^{+\infty} p_{-j}(x^\prime, \xi^\prime)$. Then we obtain
 $$ a(x^\prime) \sum_{j=0}^{+ \infty} p_{-j}(x^\prime, \xi^\prime)\sim 
     \sum_{|\alpha|=0}^{+\infty} \frac{1}{\alpha!}\partial_{\xi^\prime}^{\alpha}
     \Big(\sum_{j=0}^{+ \infty} p_{-j}(x^\prime, \xi^\prime)\Big)
      D_{x^\prime}^\alpha a(x^\prime).$$
The components of zero order coincide, 
since $a$ is scalar-valued. Equality of the components of order $-1$ means 
 \begin{align}
  \label{eq:simbolo1}
  \sum_{|\alpha|=1}\partial_{\xi^\prime}^\alpha p_0(x^\prime, \xi^\prime) 
   D^{\alpha}_{x^\prime} a(x^\prime)=(\nabla_{\xi^\prime} p_0(x^\prime, \xi^\prime),\nabla a(x^\prime))=0.
 \end{align}
Since both $a$ and $a^*$ preserve the domain also $a+a^*$ and $(a-a^*)/i$ preserves the domain.
Therefore, it is not a restriction to suppose that $a$ is real valued.
Since $p_0$ is not constant,
the following Lemma \ref{lem:const} implies $\nabla a=0$, i.e., $a$ is locally constant on the boundary. 

In the odd case, if $B$ is not invertible, the involved projection is $P=P_{>0}+P_{g}$. Since it differs from  
$P_{\geq}$ by a finite-dimensional, smoothing operator, the homogeneous components of $P$ coincide with 
those of $P_{\geq}$ and we can argue as above. 
\end{proof}

Notice that Proposition \ref{prop:const} holds infact true for all boundary conditions described in 
Remark \ref{rem:genproj}.

\begin{lemma}\label{lem:const}
Let $q\in\scrC^1(\rz^m\setminus\{0\})$ be positively homogeneous of degree $0$, $v\in\rz^m$, and 
 $$(\nabla q(\eta),v)=0\qquad\forall\;\eta\not=0.$$
Then either $v=0$ or $q\equiv\mathrm{const}$. 
\end{lemma}
\begin{proof}
Assume $v\not=0$. Let $V$ denote the span of $v$. Then, for arbirary $\eta\in\rz^m\setminus V$, 
 $$\frac{d}{dt}q(\eta+tv)=(\nabla q(\eta+tv),v)=0 \qquad\forall\;t\in\rz.$$
Thus $t\mapsto q(\eta+tv)$ is constant in $t$. Hence, using the homogeneity, 
 $$q(\eta)=q(\eta+tv)=q(\eta/t+v)\xrightarrow{t\to+\infty}q(v)\qquad\forall\;\eta\in\rz^m\setminus V.$$
By continuity of $q$ on $\rz^m\setminus\{0\}$, it follows that $q\equiv q(v)$.  
\end{proof}

\begin{theorem}\label{thm:closure}
Let $(\calA^\infty_{\calD},\calH,\calD)$  be the spectral triple associated with a Dirac operator $D$ and 
$\aps$-type boundary conditions as described in the previous Subsection $\mathrm{\ref{sec:4.1}}$. 
Then the closure of $\calA^\infty_\calD$ in 
$\scrL(\calH)$ is isomorphic to 
 $$\scrC_\partial(\Omega)
     :=\big\{a\in\scrC(\Omega)\mid a|_{\partial\Omega}\text{ is locally constant}\big\}.\footnotemark$$
\end{theorem}
\begin{proof}
Let\footnotetext{Here, we identify functions from $\scrC(\Omega)$ with their operators of multiplication; 
the operator norm as an element in $\scrL(L^2(\Omega,E))$ then coincides with the supremum norm of the function. 
Hence taking the closure refers to uniform convergence on $\Omega$.} 
us introduce the space $V$ consisting of those 
functions $a\in \scrC^\infty(\Omega)\cap\scrC_\partial(\Omega)$ such that $a-a|_{\partial\Omega}$ vanishes 
to infinite order at the boundary. Then, by Lemmas \ref{thm:1carA} and \ref{prop:const}, 
$V\subset \calA^\infty_\calD\subset \scrC_\partial(\Omega)$. 
However, it is an elementary fact that the closure of $V$ with repect to the supremum norm coincides with 
$\scrC_\partial(\Omega)$. 
\end{proof}

Denoting by $\wh\Omega$ the topological space obtained from $\Omega$ by collapsing 
each connected component of $\partial\Omega$ to a separate, single point, $\scrC_\partial(\Omega)$ 
is isomorphic to $\scrC(\wh{\Omega})$. In this sense, the spectral triple constructed above does not see the 
boundary of the manifold.

\subsection{Example: A spectral triple on the disk}
\label{sec:exe}

It is natural to ask which of the hypotheses of Connes' reconstruction theorem are not met 
when considering the regular spectral triple of a manifold with boundary 
$(\calA_{\calD}^\infty, \calH, \calD)$ introduced above. The answer is that the so called 
finiteness axiom is violated. Indeed, it is not true that $\calH^{\infty}_{\calD}$  is a finitely generated 
projective $\calA_{\calD}^{\infty}$-module. Indeed, if this were true, then
 \[
   \calH^{\infty}_{\calD} \cong p(\calA_{\calD}^\infty\oplus\ldots\oplus \calA_{\calD}^\infty) 
   \qquad\text{($N$ summands)}
 \]
for a suitable $N$ and a suitable projection $p$ given by an $N\times N$-matrix with entries 
from $\calA_{\calD}^\infty$. Therefore, in view of Theorem $\mathrm{\ref{thm:closure}}$, 
restricting $\calH^{\infty}_{\calD}$ to the boundary would give a finite-dimensional space. 
However, as we shall verify in an explicit example, this is not true in general. 
 
Following the exposition in \cite{FGS98}, we consider on
$\mathbb{B}:=\{(x,y)\in\rz^2\mid x^2+y^2\le 1\}$ the Dirac operator 
 $$D= i \begin{pmatrix}0 &1\\1 & 0\end{pmatrix} \partial_x
          + i\begin{pmatrix} 0 &-i\\ i & 0\end{pmatrix} \partial_y,$$
acting on $\cz^2$-valued functions. Passing to polar coordinates 
$(r, \theta)$,\footnote{In the literature sometimes a different sign convention is used, cf.\ \cite{CC13} 
for example. This is due to the change of coordinates $\theta\to\frac{\pi}{2}-\tilde{\theta}$.}
 $$D= i \begin{pmatrix}0 &e^{- i \theta}\\ e^{i \theta} & 0\end{pmatrix} \partial_r
    + \frac{i}{r}\begin{pmatrix} 0 &-i e^{- i \theta}\\ i e^{i \theta} & 0 \end{pmatrix} \partial_\theta,$$
respectively 
 $$ D= 
   \begin{pmatrix}  
    0 & i e^{-i \theta} \pt{\partial_r+ B(r)}\\
    -i e^{i \theta} \pt{-\partial_r+ B(r)}& 0 
   \end{pmatrix}, 
   \qquad B(r)= -\frac{i}{r}\partial_\theta.$$
The operator $B:=B(1)$, acting as unbounded operator in $L^2(\partial\bz)$,  has spectrum consisting of 
the eigenvalues $n\in\gz$, with corresponding eigenfunctions $e^{in\theta}$. If $P_\ge$ and $P_\le$ 
denote the orthonormal projections in $L^2(\partial\bz)$ onto the span of the $e^{in\theta}$ with $n\ge0$ 
and $n\le0$, respectively, then $\calD:=D_{\mathrm{APS}}$ has domain 
 $$\dom\calD=\big\{\psi=(\psi_1,\psi_2)\in H^1(\bz,\cz^2)\mid 
     P_\ge\gamma_0\psi_1=P_\le\gamma_0\psi_2=0\big\},$$
where $\gamma_0u=u|_{\partial\bz}$ is the restriction to the boundary of $\bz$. 

\begin{proposition} 
The restriction of $\calH^\infty_\calD$ to the boundary of $\mathbb{B}$ is an infinite-dimensional space. 
\end{proposition}
\begin{proof}
With an integer $k>0$ let us consider $\psi_k=(\psi_{k,1},0)\in\scrC^\infty(\bz,\cz^2)$ with 
\begin{equation}\label{eq:contrk}
    \psi_{k,1}(r,\theta)= \chi(r)\pt{\cosh(-k \log r ) + \sinh (-k \log r)} e^{-ik \theta}, 
\end{equation}
where $\chi$ is a smooth function identically equal to $1$ near $r=1$
and identically equal to zero near $r=0$. Obviously, $\psi_k|_{\partial\mathbb{B}}=(e^{-ik\theta},0)$ and 
$\psi_k$ belongs to the domain of $\calD$. 
Moreover, $D\psi_k=(0,\varphi_{k,2})$ with 
$$\varphi_{k,2}(r,\theta)= -i e^{i \theta} \Big( - \partial_r-\frac{i}{r} \partial_\theta \Big) 
    \left[   \chi(r)   \pt{\cosh(-k \log r ) + \sinh (-k \log r)} e^{-ik \theta}\right].$$
A straight-forward calculation now reveals that 
$$\varphi_{k,2}(r,\theta)=
    i e^{i \theta}\pt{\partial_r \chi}(r)\Big(\big(\cosh(-k \log r ) + \sinh (-k \log r)\big) e^{-ik \theta}\Big).$$
Therefore $D\psi_k$ is supported in the interior of $\bz$, since $\partial_r \chi$ vanishes near $r=1$. 
Since $D$ is a differential operator, this is then also true for $D^n \psi_k$ for every $n\ge 1$, hence 
the $\aps$-boundary conditions are trivially fulfilled.
\end{proof}

\section{Appendix: Ellipticity and Fredholm property}\label{sec:appendix}

Let $\Sigma\subset\rz^n$ and assume we have two families of Banach spaces 
 $$\{H^s_j\}_{s\in\Sigma},\qquad j=0,1,$$
having the following properties, for every $s$: 
\begin{enumerate}
 \item $H^\infty_j:=\mathop{\cap}_{s\in\Sigma}H^s_j$ is a dense subspace of $H^s_j$.  
 \item Any continuous operator $T:H^s_j\to H^s_j$ with $\mathrm{im}\,T\subset H^\infty_j$ is compact. 
\end{enumerate}

\subsection{Invariance of the index}
Let us consider two operators 
 $$A_j:H^\infty_j\lra H^\infty_{1-j},\qquad j=0,1,$$
that extend by continuity to operators 
 $$A_j^s:H^s_j\lra H^s_{1-j},\qquad s\in\Sigma,$$
and such that 
 $$C_j:=1-A_{1-j}A_j:H^\infty_j\lra H^\infty_j$$
are regularizing operators, in the sense that the extensions $C_j^s$ satisfy 
 $$\mathrm{im}\,C_{j}^s\subset H^\infty_j,\qquad s\in\Sigma.$$ 
Due to assumption $(2)$ on the compactness, each $A_j^s$ is a Fredholm operator. 
We shall see, in particular, that the corresponding index does not depend on $s\in\Sigma$. 
This has already been observed in Lemma 1.2.94 of \cite{KaSc} in case of one-parameter scales of 
Hilbert spaces $($requiring, in particular, continuous  embeddings $H^s\hookrightarrow H^t$ for $s\ge t$, 
which are compact in case $s>t)$; the proof we give here extends to multi-parameter families of Banach 
spaces.  

\begin{example}
In connection with boundary value problems, typical families arising are of the form 
$\Sigma=\big\{s=(r,p)\mid p>1,\;r>1/p\big\}\subset\rz^2$ and 
 $$H^s_j=H^r_p(\Omega,E_j)\oplus B^{r-1/p}_{pp}(\pO,F_j),\qquad s=(r,p)\in\Sigma,$$
with a smooth compact manifold with boundary $\Omega$ and vector-bundles $E_j$ and $F_j$ 
$($direct sum of Sobolev $($Bessel potential$)$ and Besov spaces$)$. Then 
 $$H^\infty_j=\scrC^\infty(\Omega,E_j)\oplus \scrC^\infty(\pO,F_j)$$
and $(1)$, $(2)$ hold due to well-known embedding theorems. 
\end{example}

Let us first observe the following consequence, refered to as \textit{elliptic regularity} in the sequel: 
Let $f\in H^\infty_1$ and $A^s_0u=f$ with $u\in H^s_0$ for some $s\in\Sigma$. 
Then $u\in H^\infty_0$ and $f=A_0u$. In fact,  
 $$A_1f=A^s_1f=A_1^sA_0^su=(1-C^s_0)u=u-C^s_0u$$
shows that $u=A_1f+C^s_0u$ belongs to $H^\infty_0$. 

\begin{lemma}\label{lem:complement}
Let $V_1$ be a finite-dimensional subspace of $H^\infty_1$ such that 
\begin{itemize}
 \item[$(1)$] $V_1\cap\mathrm{im}\,A_0^s=\{0\}$,\qquad\quad 
  $(2)$\; $V_1+\mathrm{im}\,A_0^s=H^s_1$
\end{itemize}
for some fixed value $s=s_0$. Then both $(1)$ and $(2)$ hold for arbitrary $s\in\Sigma$. 
\end{lemma}
\begin{proof}
$(1)$ is a direct consequence of elliptic regularity: If $f=A^s_0u\in V_1$ then $u\in H^\infty_0$ and, in particular, 
$f=A^{s_0}_0u\in V_1\cap\mathrm{im}\,A_0^{s_0}=\{0\}$. 

For $(2)$ let $f\in H^s_1$ with some $s\in\Sigma$ be given. Choose a sequence $(f_n)\subset H^\infty_1$ 
converging to $f$ in $H^s_1$. By $(2)$ for $s=s_0$ and elliptic regularity, we can write 
 $$f_n=v_n+A_0u_n,\qquad v_n\in V_1,\quad u_n\in H^\infty_0.$$
Since $A^s_0$ is a Fredholm operator, $\mathrm{im}\,A^s_0$ is a closed subspace of $H^s_1$. By $(1)$ 
it has a complement of the form $V_1\oplus V$ with some finite-dimensional $V$; this is then a topological 
complement. It follows that $(f_n)$ converges in $H^s_1$ if, and only if, both $(A^s_0u_n)$ and $(v_n)$ 
converge in $H^s_1$. Thus there exists a $u\in H^s_0$ and a $v\in V_1$ such that, in $H^s_1$,  
 $$f_n\xrightarrow{n\to+\infty}f,\qquad f_n=v_n+A^s_0 u_n\xrightarrow{n\to+\infty} v+A^s_0u.$$
Therefore $f=v+A^s_0u\in V_1+\mathrm{im}\,A_0^s$.  
\end{proof}

\begin{proposition}\label{prop:index}
Under the above assumptions there exist finite-dimensional subspaces $V_j\subset H^\infty_j$ 
such that, for every $s\in\Sigma$, 
 $$\mathrm{ker}\,A^s_0=V_0,\qquad H^s_1=V_1\oplus\mathrm{im}\,A^s.$$ 
In particular, $\mathrm{ind}\,A^s_0=\mathrm{dim}\,V_0-\mathrm{dim}\,V_1$ does not depend on $s$. 
Moreover, if $A_0^s$ is invertible for some $s\in\Sigma$ then it is for all $s$. 
\end{proposition}
\begin{proof}
By elliptic regularity, $V_0:=\mathrm{ker}\,A^s_0$ is independent of $s$. Due to the Fredholm property it is 
finite-dimensional. 

By the previous Lemma \ref{lem:complement} it suffices to find a subspace $V_1\subset H^\infty_1$ that is a 
complement to $\mathrm{im}\,A^s_0$ for some fixed choice of $s$. 
Let $W=\mathrm{span}\{w_1,\ldots,w_n\}$ be a complement to $\mathrm{im}\,A^s_0$ in $H^s_1$. Write 
 $$w_k=(C^s_1+A^s_0A^s_1)w_k=v_k+ A^s_0u_k,\qquad k=1,\ldots,n,$$ 
with $v_k:=C^s_1w_k\in H^\infty_1$ and $u_k:=A^s_1w_k\in H^s_0$. Then  
$V_1=\mathrm{span}\{v_1,\ldots,v_n\}$ is the desired complement. 
\end{proof}

\subsection{Inverses modulo projections}

Now let $L_{jk}^0$, $j,k\in\{0,1\}$, denote certain vector spaces of operators $H^\infty_j\to H^\infty_{k}$, 
whose elements extend by continuity to operators in $\scrL(H^s_j,H^s_k)$ for every $s\in\Sigma$. 
Let $L_{jk}^{-\infty}$ be subspaces of regularizing $($in the sense described above$)$ operators. 
Assume that, for every choice of $j,k,\ell\in\{0,1\}$, compostion of operators induces maps 
 $$(A,B)\mapsto AB,\quad L_{k\ell}^s\times L^{t}_{jk}\lra L^{s+t}_{j\ell},
     \qquad s,t\in\{-\infty,0\}.$$
By definition, a parametrix to $A_0\in L^0_{01}$ is any operator $A_1\in L^0_{10}$ such that 
 $$1-A_{10}A_{01}\in L^{-\infty}_{00}\quad\text{and}\quad 1-A_{01}A_{10}\in L^{-\infty}_{11}.$$
An operator possessing a parametrix is called elliptic. 

\begin{proposition}
Let $A_0\in L^0_{01}$ be elliptic and $V_j\subset H^\infty_j$ be two spaces as described in 
Proposition $\mathrm{\ref{prop:index}}$. Assume that $\pi_j\in L^{-\infty}_{jj}$ are projections with 
$\mathrm{im}\,\pi_j=V_j$. Furthermore assume that, for every $s\in\Sigma$, compostion of operators 
induces maps 
\begin{equation}\label{eq:sandwich}
 (A,B,C)\mapsto ABC,\quad L_{00}^{-\infty}\times \scrL(H^s_{1},H^s_0)\times L_{11}^{-\infty}\lra L_{10}^{-\infty},
\end{equation}
i.e., sandwiching a continuous operator between two smoothing operators results in a smoothing operator. 
Then there exists a parametrix $A_1\in L^0_{10}$ such that 
 $$A_1A_0=1-\pi_0,\qquad A_0A_1=1-\pi_1.$$
\end{proposition}
\begin{proof}
Fix some $s\in\Sigma$. Then $\pi_j^s$ are projections in $H^s_j$ 
with image $V_j$. Note that $A_0^s:\mathrm{ker}\,\pi_0^s\to \mathrm{im}\,A_0^s$ is bijective, hence has an 
inverse. Let $A_1^s\in\scrL(H_1^s,H_0^s)$ be the operator acting like this inverse on $\mathrm{im}\,A_0^s$ 
and vanishing on $V_1$. By construction we thus have 
 $$A_1^sA_0^s=1-\pi_0^s,\qquad A_0^sA_1^s=1-\pi_1^s.$$
Now let $\wt{A}_1$ be a parametrix to $A_0$ and $C_0=1-\wt{A}_1A_0$ and $C_1=1-A_0\wt{A}_1$. Then 
\begin{align*}
 A_1^s-\wt{A}_1^s
 &=\big(\wt{A}_1^sA_0^s+C_0^s\big)\big(A_1^s-\wt{A}_1^s\big)
   =\wt{A}_1^s\big(C_1^s-\pi_1^s\big)+C_0^s\big(A_1^s-\wt{A}_1^s\big)\\
 A_1^s-\wt{A}_1^s
 &=\big(A_1^s-\wt{A}_1^s\big)\big(A_0^s\wt{A}_1^s+C_1^s\big)
   =\big(C_0^s-\pi_0^s\big)\wt{A}_1^s+\big(A_1^s-\wt{A}_1^s\big)C_1^s.
\end{align*}
Substituting the second equation into the first yields 
 $$A_1^s-\wt{A}_1^s=\wt{A}_1^s\big(C_1^s-\pi_1^s\big)+
       C_0^s\big(C_0^s-\pi_0^s\big)\wt{A}_1^s+C_0^s\big(A_1^s-\wt{A}_1^s\big)C_1^s.$$
Hence the desired parametrix is 
 $$A_1=\wt{A}_1+\wt{A}_1\big(C_1-\pi_1\big)+
       C_0\big(C_0-\pi_0\big)\wt{A}_1+C_0\big(A_1^s-\wt{A}_1^s\big)C_1,$$
since the last term belongs to $L^{-\infty}_{10}$ due to assumption \eqref{eq:sandwich}. 
\end{proof}

\textbf{Acknowledgements: }
The authors thank J.\ Aastrup, M.\ Goffeng, C.\ Levy and E.\ Schrohe for valuable discussions on spectral triples.
The first author has been partially supported by the Gruppo Nazionale per l'Analisi
Matematica, la Probabilit\`a e le loro Applicazioni (GNAMPA) of the Istituto
Nazionale di Alta Matematica (INdAM).


\end{document}